\documentclass[10pt]{article}
\font\smallit=cmti10

\usepackage{latexsym,epsfig} %% Add other packages as necessary
\usepackage{amsmath}
\usepackage{amsthm, amssymb}
\usepackage{mathtools}
\usepackage{graphicx}   % need for figures
\usepackage{url}
\usepackage[utf8]{inputenc} % Swedish
\usepackage{mathrsfs} % for mathscr
\usepackage{ytableau} % For Young tableau/diagram
\usepackage{tikz}
\usetikzlibrary{arrows}

\makeatletter

\renewcommand\section{\@startsection {section}{1}{\z@}
{-30pt \@plus -1ex \@minus -.2ex}
{2.3ex \@plus.2ex}
{\normalfont\normalsize\bfseries}}

\renewcommand\subsection{\@startsection{subsection}{2}{\z@}
{-3.25ex\@plus -1ex \@minus -.2ex}
{1.5ex \@plus .2ex}
{\normalfont\normalsize\bfseries}}

\renewcommand{\@seccntformat}[1]{\csname the#1\endcsname. }

\makeatother

\newtheorem{theorem}{Theorem}
\newtheorem{lemma}{Lemma}
\newtheorem{conjecture}{Conjecture}

\newtheorem*{chernoff}{Chernoff Bound}

\DeclareMathOperator{\ord}{\textup{\textsf{ord}}}

\newcommand{\Par}{\mathcal{P}}
\newcommand{\Bin}{\textup{Bin}}
\newcommand{\rescaled}[1]{\partial{#1}}
\newcommand{\rescaleds}[2]{\partial^{#1}{#2}}
\newcommand{\Rnn}{\mathbb{R}_{\geq 0}}

\begin{document}

\begin{center}
\uppercase{\bf An exponential limit shape of random \lowercase{$q$}-proportion Bulgarian solitaire}
\vskip 20pt
{\bf Kimmo Eriksson}\\
{\smallit Mälardalen University, School of Education, Culture and Communication,\\Box 883, SE-72123 Västerås, Sweden}\\
{\tt kimmo.eriksson@mdh.se}\\
\vskip 10pt
{\bf Markus Jonsson\footnote{Corresponding author}}\\
{\smallit Mälardalen University, School of Education, Culture and Communication,\\Box 883, SE-72123 Västerås, Sweden}\\
{\tt markus.jonsson@mdh.se}\\
\vskip 10pt
{\bf Jonas Sjöstrand}\\
{\smallit Royal Institute of Technology, Department of Mathematics,\\SE-10044 Stockholm, Sweden}\\
{\tt jonass@kth.se}\\
\end{center}
\vskip 30pt

%\centerline{\smallit Received: , Revised: , Accepted: , Published: } % We will fill in the dates
\centerline{\smallit \today}

\vskip 30pt

\centerline{\bf Abstract}

\noindent
%Put your abstract here. Please limit it to half of a page of text.
We introduce \emph{$p_n$-random $q_n$-proportion Bulgarian solitaire} ($0<p_n,q_n\le 1$), played on $n$ cards distributed in piles. In each pile, a number of cards equal to the proportion $q_n$ of the pile size rounded upward to the nearest integer are candidates to be picked. Each candidate card is picked with probability $p_n$, independently of other candidate cards. This generalizes Popov's random Bulgarian solitaire, in which there is a single candidate card in each pile. Popov showed that a triangular limit shape is obtained for a fixed $p$ as $n$ tends to infinity. Here we let both $p_n$ and $q_n$ vary with $n$. We show that under the 
conditions $q_n^2 p_n n/{\log n}\rightarrow \infty$ and $p_n q_n \rightarrow 0$ as $n\to\infty$, the $p_n$-random $q_n$-proportion Bulgarian solitaire has an exponential limit shape. 

\pagestyle{myheadings}
%\markright{\smalltt INTEGERS: 15 (2015)\hfill}
\thispagestyle{empty}
\baselineskip=12.875pt
\vskip 30pt

\section{Introduction}
\label{sec:bulgsol}
The game of Bulgarian solitaire has received a great deal of attention, see reviews by Hopkins \cite{Hopkins2011} and Drensky \cite{drensky2015bulgarian}. The Bulgarian solitaire is played with a deck of $n$ identical cards divided arbitrarily into a number of piles. A move consists of picking a card from each pile and letting these cards form a new pile. If piles are sorted in order of decreasing size, every position in the solitaire is equivalent to a Young diagram of an integer partition of $n$.   

Popov \cite{Popov} considered a random version of Bulgarian solitaire defined by a probability $p\in(0,1]$, such that one card from each pile is picked with probability $p$, independently of the other piles. We will refer to this stochastic process on configurations as \emph{$p$-random} Bulgarian solitaire. The probabilities of configurations converge to a stationary distribution. Popov showed that as $n$ grows to infinity and configuration diagrams are downscaled by $\sqrt{n}$ in both dimensions, the stationary probability of the set of configurations that deviate from a triangle with slope $p$ by more than $\varepsilon>0$ tends to zero. In this sense, random configurations has a limit shape.

The objective of the present paper is to study such limit shapes in a generalization of random Bulgarian solitaire.

\subsection{$q_n$-proportion Bulgarian solitaire}
\label{sec:qn-Bulgarian}
Olson \cite{Olson2016} introduced a generalization of Bulgarian solitaire in which the number of cards that are picked from a pile of size $h$ is given by some non-negative valued function $\sigma(h)$. Eriksson, Jonsson and Sj{\"o}strand \cite{EJS} recently studied the special case when $\sigma$ is well-behaved in the sense that $\sigma(1)=1$ and both $\sigma(h)$ and $h-\sigma(h)$ are non-decreasing functions of $h$. In particular, they studied a special case that they called \emph{$q_n$-proportion Bulgarian solitaire}, defined by the rule $\sigma(h)=\lceil q_n h\rceil$. This means that from each pile we pick a number of  cards given by the proportion $q_n$ of the pile size rounded upward to the nearest integer. To illustrate the effect of the parameter $q_n$, set it to $0.3$ and consider the configuration $(6,2,2,1)$. From the first pile we pick $\lceil 0.3\times 6\rceil = 2$  cards; similar calculations give that 1 card is picked from each of the other three piles. Note that for $q_n\le 1/n$ exactly one card is always picked from each pile, retrieving the ordinary Bulgarian solitaire. 

As $n$ tends to infinity, Eriksson, Jonsson and Sj{\"o}strand \cite{EJS} determined limit shapes of stable configurations of  $q_n$-proportion Bulgarian solitaire: In case ${q_n^2 n\rightarrow 0}$, the limit shape is triangular, which generalizes the limit shape result for the ordinary Bulgarian solitaire. For other asymptotic behavior of $q_n$, other limit shapes were obtained. Specifically, in case $q_n^2 n\rightarrow \infty$, the limit shape is exponential. The intermediate case $q_n^2 n\rightarrow C>0$ produces a family of limit shapes  interpolating between the triangular and the exponential shape.

\subsubsection{$p_n$-random $q_n$-proportion Bulgarian solitaire}
We shall examine a $p_n$-random version of $q_n$-proportion Bulgarian solitaire, in which the proportion $q_n$ (rounded upward) of cards in a pile are only \emph{candidates} to be picked, each of which is picked only with probability $p_n$, independently of all other candidate cards. This process will be denoted by $\mathscr{B}(n,p_n,q_n)$. Note that in the special case of a fixed $p$ and for $q_n \le 1/n$, this process is equivalent to Popov's $p$-random Bulgarian solitaire. 

Our focus will be on establishing a regime in which $p_n$-random $q_n$-proportion Bulgarian solitaire has an exponential limit shape.

\section{The concept of limit shapes}\label{sc:limit-concept}
In this section we give the precise definitions of the limit shapes we consider. 
Let $\mathcal{P}(n)$ be the set of integer partitions of $n$.
For any partition $\lambda\in\mathcal{P}(n)$ with $N=N(\lambda)$ positive parts $\lambda_1 \ge \lambda_2 \ge \dotsc \ge \lambda_N>0$, define $\lambda_i=0$ for $i>N(\lambda)$, and the \emph{diagram} of $\lambda$ as the Young diagram oriented such that the parts of $\lambda$ are represented by left and bottom aligned columns, weakly decreasing in height from left to right. For example, \!\!\!\!
\ytableausetup{boxsize=5pt, aligntableaux=bottom}
\begin{ytableau}
	\hspace{0pt} \\
	\hspace{0pt} & & \\	
\end{ytableau}
is the diagram of the partition $(2,1,1)$.
We define the \textit{diagram-boundary function} of $\lambda$ as the nonnegative, weakly decreasing and piecewise constant function $\partial\lambda:\Rnn\rightarrow\Rnn$ describing the boundary of $\lambda$, given by
\[
\partial\lambda(x)=\lambda_{\lfloor x \rfloor+1}.
\]
Following \cite{Eriksson2012575} and \cite{VershikStatMech}, the diagram is downscaled using some \emph{scaling factor} ${a_n>0}$ such that all row lengths are multiplied by $1/a_n$ and all column heights are multiplied by $a_n/n$, yielding a constant area of 1. Following \cite{EJS}, we shall consistently make the choice $a_n = n/\lambda_1$, such that the height of the diagram is scaled to 1.
 
Thus, given a partition $\lambda$, define the \emph{rescaled diagram-boundary function} of $\lambda$ as the nonnegative, real-valued, weakly decreasing and piecewise constant function $\rescaled{\lambda}:\Rnn\rightarrow\Rnn$ given by
\begin{equation}
\label{eq:def_rescaled}
\rescaled{\lambda}(x)=\dfrac{1}{\lambda_1}\partial\lambda(xn/\lambda_1)=\frac{1}{\lambda_1}\lambda_{\lfloor xn/\lambda_1 \rfloor+1}.
\end{equation}

The $p_n$-random $q_n$-proportion Bulgarian solitaire $\mathscr{B}(n,p_n,q_n)$ (with $p_n,q_n \in(0,1]$) can be regarded as a Markov chain on the finite state-space $\Par(n)$. Let us denote the sequence of visited states by $(\lambda^{(0)},\lambda^{(1)},\dotsc)$. In the truly random case of $p_n<1$, it is straightforward to verify that this Markov chain is aperiodic and irreducible. It is well-known that an aperiodic and irreducible Markov chain on a finite state-space has a unique stationary distribution $\pi$ and that starting from any initial state the distribution of the $i$th state $\lambda^{(i)}$ converges to $\pi$ as $i$ tends to infinity. We denote by $\pi_{n,p_n,q_n}$ the stationary measure of the Markov chain $(\lambda^{(0)},\lambda^{(1)},\dotsc)$ on $\Par(n)$ given by $\mathscr{B}(n,p_n,q_n)$ for $p_n<1$.%
\footnote{Readers acquainted with the limit shape literature may wonder whether the stationary measure has the property of being \emph{multiplicative}, in the sense of interpretable as the product measure on the space of integer sequences restricted to a certain affine subspace \cite{fristedt1993structure}. The multiplicative property is useful in limit shape problems and related problems \cite{corteel1999multiplicity, erlihson2004reversible, goh2008random, pittel1997likely, VershikStatMech}. However, such techniques will not be used here as $\pi_{n,p_n,q_n}$ is unlikely to be multiplicative in general.} When we refer to a limit shape of the process $\mathscr{B}(n,p_n,q_n)$ for $p_n<1$ as $n$ grows to infinity, we shall mean the limit shape of the stationary measure $\pi_{n,p_n,q_n}$. The intuitive sense of this concept is that \emph{when the solitaire is played on a sufficiently large number of cards for sufficiently long the configuration will almost surely be very close to the limit shape after suitable downscaling}. Following Vershik \cite{VershikStatMech}, a sequence $\{\pi_n \}$ of probability distributions on $\Par(n)$ is said to have a limit shape $\phi$  if the downscaled diagrams approach $\phi$ in probability as $n$ grows to infinity. The exact condition for convergence can vary. Consistent with 
Yakubovich \cite{yakubovich2012ergodicity} and Eriksson and Sjöstrand \cite{Eriksson2012575},
we shall use the definition that
\begin{equation}
\label{eq:limitshape_def}
%\lim_{n\rightarrow\infty}\pi_n\left\lbrace \lambda\in\Par(n): %\|\rescaleds{a_n}{\lambda}(x)-\phi(x)\|_{\infty}<\varepsilon\right\rbrace=1,
\lim_{n\rightarrow\infty}\pi_n\left\lbrace \lambda\in\Par(n): |\rescaleds{a_n}{\lambda}(x)-\phi(x)|<\varepsilon\right\rbrace=1
\end{equation}
%where $\|\cdot\|_{\infty}$ denotes the max-norm $\| f \|_{\infty}=\sup \big\{|f(x)|:x\geq 0\big\}$.
for all $x>0$ and all $\varepsilon>0$.%
\footnote{Vershik \cite{VershikStatMech} and Erlihson and Granovsky \cite{erlihson2008limit} used a stronger condition for convergence toward a limit shape, namely that
\[
\lim_{n\rightarrow\infty}\pi_n\left\lbrace \lambda\in\Par(n): \sup_{x\in[a,b]} |\rescaleds{a_n}{\lambda}(x)-\phi(x)|<\varepsilon\right\rbrace=1
\]
should hold for any compact interval $[a,b]$, and any $\varepsilon>0$.
%Yakubovich \cite{yakubovich2012ergodicity} and Eriksson and Sjöstrand \cite{Eriksson2012575} used an even weaker condition:
%\[
%\lim_{n\rightarrow\infty}\pi_n\left\lbrace \lambda\in\Par(n): %|\rescaleds{a_n}{\lambda}(x)-\phi(x)|<\varepsilon\right\rbrace=1
%\]
%for all $x>0$ and all $\varepsilon>0$.
} 

\section{The approach of ordering piles by time of creation}
\label{sec:repr_compositions}
It will sometimes be useful to explicitly order piles by time of creation rather than by size. Here we develop this approach.

When parts are not sorted by size, a configuration is not represented by an integer partition but by a \emph{weak integer composition}: an infinite sequence $\alpha=(\alpha_1,\alpha_2,\dotsc)$, not necessarily decreasing, of nonnegative integers adding up to $n$. Let $\mathcal{W}(n)$ denote the set of weak compositions of $n$. 
We define the \textit{diagram}, the \textit{diagram-boundary function} $\partial\alpha$, and the \textit{rescaled diagram-boundary function} $\rescaled{\alpha}$ of a weak  composition $\alpha$ in exact analogy to the way we defined them for partitions in Section~\ref{sc:limit-concept}. For example, the diagram of $\alpha=(3,0,2,4,1,0,0,\dotsc)$ and the corresponding function graph $y=\partial\alpha(x)$ are shown in Figure~\ref{fig:psi_alpha_example}.
Also, for a weak composition $\alpha=(\alpha_1,\alpha_2,\dotsc,\alpha_N,0,0,\dotsc)$ we define the number of parts $N=N(\alpha)$ disregarding the trailing zeros.
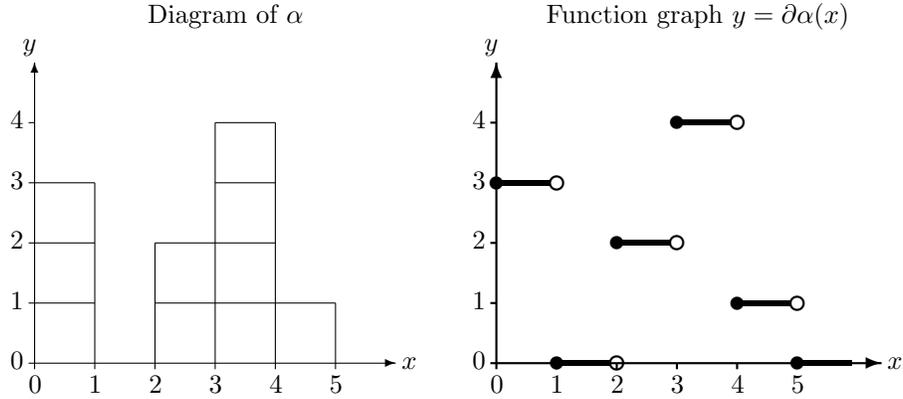
\begin{figure}[ht]
\setlength{\unitlength}{0.08cm}
\centering

\begin{tabular}{cc}
Diagram of $\alpha$ & Function graph $y=\partial\alpha(x)$ \\
\begin{picture}(70,60)
%\linethickness{1pt}
\put(4,5){\vector(0,1){50}}
\put(4,5){\vector(1,0){60}}
\put(3,0){$0$}
\put(4,4){\line(0,1){1}}
\put(13,0){$1$}
\put(14,4){\line(0,1){1}}
\put(23,0){$2$}
\put(24,4){\line(0,1){1}}
\put(33,0){$3$}
\put(34,4){\line(0,1){1}}
\put(43,0){$4$}
\put(44,4){\line(0,1){1}}
\put(53,0){$5$}
\put(54,4){\line(0,1){1}}

\put(0,4){$0$}
\put(3,5){\line(1,0){1}}
\put(0,14){$1$}
\put(3,15){\line(1,0){1}}
\put(0,24){$2$}
\put(3,25){\line(1,0){1}}
\put(0,34){$3$}
\put(3,35){\line(1,0){1}}
\put(0,44){$4$}
\put(3,45){\line(1,0){1}}
%\linethickness{1pt}
\put(14,5){\line(0,1){30}}
\put(24,5){\line(0,1){20}}
\put(34,5){\line(0,1){40}}
\put(44,5){\line(0,1){40}}
\put(54,5){\line(0,1){10}}

\put(4,15){\line(1,0){10}}
\put(24,15){\line(1,0){30}}
\put(4,25){\line(1,0){10}}
\put(24,25){\line(1,0){20}}
\put(34,35){\line(1,0){10}}
\put(34,45){\line(1,0){10}}
\put(4,35){\line(1,0){10}}

\put(65,4){$x$}
\put(2,57){$y$}
\end{picture}
\rule{0pt}{60pt} & 
\begin{picture}(75,60)
\thicklines
%\linethickness{1pt}
\put(4,5){\vector(0,1){50}}
\put(4,5){\vector(1,0){64}}

\put(3,0){$0$}
\put(4,4){\line(0,1){1}}
\put(13,0){$1$}
\put(14,4){\line(0,1){1}}
\put(23,0){$2$}
\put(24,4){\line(0,1){1}}
\put(33,0){$3$}
\put(34,4){\line(0,1){1}}
\put(43,0){$4$}
\put(44,4){\line(0,1){1}}
\put(53,0){$5$}

\put(0,4){$0$}
\put(3,5){\line(1,0){1}}
\put(0,14){$1$}
\put(3,15){\line(1,0){1}}
\put(0,24){$2$}
\put(3,25){\line(1,0){1}}
\put(0,34){$3$}
\put(0,44){$4$}
\put(3,45){\line(1,0){1}}

\linethickness{2pt}
\put(4,35){\circle*{2}}
\put(14,35){\circle{2}}
\put(4,35){\line(1,0){9}}

\put(14,5){\circle*{2}}
\put(24,5){\circle{2}}
\put(14,5){\line(1,0){9}}

\put(34,45){\circle*{2}}
\put(44,45){\circle{2}}
\put(34,45){\line(1,0){9}}

\put(24,25){\circle*{2}}
\put(34,25){\circle{2}}
\put(24,25){\line(1,0){9}}

\put(44,15){\circle*{2}}
\put(54,15){\circle{2}}
\put(44,15){\line(1,0){9}}

\put(54,5){\circle*{2}}
\put(54,5){\line(1,0){9}}

\put(69,4){$x$}
\put(2,57){$y$}
\end{picture}
\end{tabular}
\caption{The composition $\alpha=(3,0,2,4,1,0,0,\dotsc)\in\mathcal{W}(10)$.\label{fig:psi_alpha_example}}
\end{figure}

%Moreover, for a distribution $\pi_n$ on $\mathcal{W}(n)$ we define the \emph{average shape} as the piecewise constant function $\partial_E:\mathbb{R}_{>0} \rightarrow \mathbb{R}_{\ge 0}$ given by
%\[
%\partial_E(x) = E\alpha_{\lfloor x \rfloor + 1},
%\]
%where $\alpha\in\mathcal{W}(n)$ is sampled from $\pi_n$.

\subsection{Connecting the limit shapes of compositions and partitions}
\label{sec:thelemma}
We shall now connect compositions with partitions. For any $\alpha\in\mathcal{W}(n)$, define the operator $\ord$ as the ordering operator that arranges the parts of $\alpha$ in descending order, thus yielding a partition. We shall now prove that such sorting of the piles by size respects the convergence to a limit shape. The proof uses some basic theory of symmetric-decrea\-sing rearrangements, see for example  \cite[Ch.~10]{inequalities} or \cite[Ch.~3]{analysis}. For any measurable function $f\colon\mathbb{R}\rightarrow\Rnn$ such that $\lim_{x\rightarrow\pm\infty}f(x)=0$, there is a
unique function $f^\ast\colon\mathbb{R}\rightarrow\Rnn$, called the \emph{symmetric-decreasing rearrangement} of $f$,
with the following properties:
\begin{itemize}
\item $f^\ast$ is symmetric, that is, $f^\ast(-x)=f^\ast(x)$ for all $x$,
\item $f^\ast$ is weakly decreasing on the interval $[0,\infty)$,
\item $f^{\ast}$ and $f$ are equimeasurable, that is,
\[
\mathcal{L}(\{ x:\;f(x)>t \})=\mathcal{L}(\{ x:\;f^\ast(x)>t \})
\]
for all $t>0$, where $\mathcal{L}$ denotes the Lebesgue measure,
\item $f^{\ast}$ is lower semi-continuous.
\end{itemize}
In particular, if $f$ is a symmetric function that is weakly decreasing and right-continuous on $[0,\infty)$ and tends to $0$ at infinity, then $f^\ast=f$.
\begin{lemma}
\label{lem:reordering}
Let $\alpha\in\mathcal{W}(n)$ be a weak composition of $n$ and let $f:\Rnn\rightarrow \Rnn$ be a right-continuous and weakly decreasing function such that $f(x)\rightarrow 0$ as $x\rightarrow\infty$. The downscaled diagram-boundary functions before and after sorting of the weak composition satisfy the inequality
\[
  \| \rescaled{\ord \alpha}-f\|_{\infty}\le
  \| \rescaled{\alpha}-f\|_{\infty}.
\]
where $\|\cdot\|_{\infty}$ denotes the max-norm $\| f \|_{\infty}=\sup \big\{|f(x)|:x\geq 0\big\}$.
%$\psi_{\textup{ord}(\alpha)}\in\mathcal{N}_f(\varepsilon,n)$.
\end{lemma}
\begin{proof}
The intuition of the lemma should be obvious from Figure~\ref{fig:reordering_proof}. To be able to use the standard machinery of symmetric rearrangements, we consider the functions
$f$, $\rescaled{\alpha}$, and $\rescaled{\ord \alpha}$ as being defined
on the entire real axis by letting $f(x)=f(|x|)$ and analogously for $\rescaled{\alpha}$, and $\rescaled{\ord \alpha}$.

Since $f(x)\rightarrow 0$ as $x\rightarrow\infty$, its symmetric-decreasing rearrangement $f^{\ast}$ is defined and,
since $f$ is weakly decreasing and lower semi-continuous, we have $f^{\ast}=f$.
Similarly, $\rescaled{\ord\alpha}(x)\rightarrow 0$ as $x\rightarrow\infty$ and is weakly decreasing, so
$(\rescaled{\ord\alpha})^{\ast}=\rescaled{\ord\alpha}$.
Moreover,  $(\rescaled{\alpha})^{\ast}=\rescaled{\ord\alpha}$ must hold because the operator $\ord$ arranges the composition parts in descending order.

Now, since symmetric rearrangements decrease $L^p$-distances for any $1\le p\le\infty$ (see for example \cite{analysis}, Section 3.4), we obtain
\[
\| \rescaled{\ord \alpha}-f\|_{\infty}=\| (\rescaled{\alpha})^{\ast}-f^{\ast}\|_{\infty}\le
\| \rescaled{\alpha}-f\|_{\infty}.
\]
\end{proof}

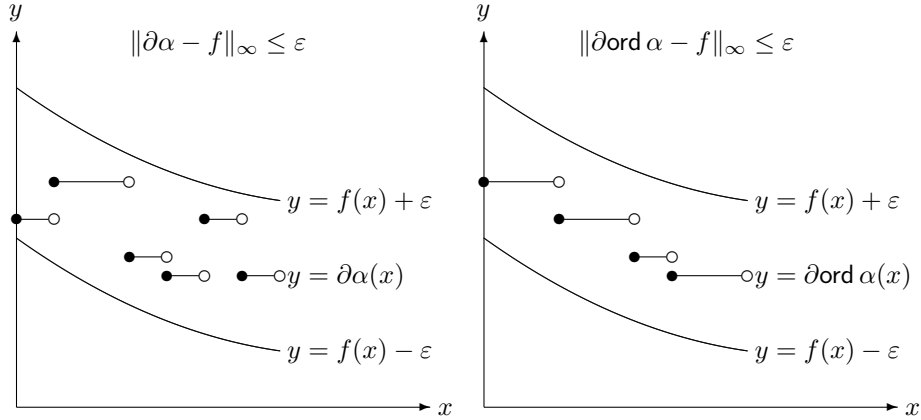
\begin{figure}[ht]
\setlength{\unitlength}{0.05cm}
\centering
\begin{tabular}{cc}
\begin{picture}(110,105)
% Coordinate system
\put(0,0){\vector(1,0){110}}
\put(0,0){\vector(0,1){100}}
\put(112,-2){$x$}
\put(-2,104){$y$}
% Curves
\qbezier(0,85)(35,60)(70,55)
\qbezier(0,45)(35,20)(70,15)
% Equations
\put(30,95){$\| \rescaled{\alpha}-f \|_{\infty} \leq \varepsilon$}
\put(72,53){$y=f(x)+\varepsilon$}
\put(72,33){$y=\rescaled{\alpha}(x)$}
\put(72,13){$y=f(x)-\varepsilon$}
% Steps
\put(1,50){\line(1,0){7.5}}
\put(0,50){\circle*{3}}
\put(10,50){\circle{3}}

\put(11,60){\line(1,0){17.5}}
\put(10,60){\circle*{3}}
\put(30,60){\circle{3}}

\put(31,40){\line(1,0){7.5}}
\put(30,40){\circle*{3}}
\put(40,40){\circle{3}}

\put(41,35){\line(1,0){7.5}}
\put(40,35){\circle*{3}}
\put(50,35){\circle{3}}

\put(51,50){\line(1,0){7.5}}
\put(50,50){\circle*{3}}
\put(60,50){\circle{3}}

\put(61,35){\line(1,0){7.5}}
\put(60,35){\circle*{3}}
\put(70,35){\circle{3}}

\end{picture}
\hspace{5pt} &
%------------------------------------------------------------
\begin{picture}(110,105)
% Coordinate system
\put(0,0){\vector(1,0){110}}
\put(0,0){\vector(0,1){100}}
\put(112,-2){$x$}
\put(-2,104){$y$}
% Curves
\qbezier(0,85)(35,60)(70,55)
\qbezier(0,45)(35,20)(70,15)
% Equations
\put(25,95){$\| \rescaled{\ord\alpha}-f \|_{\infty} \leq \varepsilon$}
\put(72,53){$y=f(x)+\varepsilon$}
\put(72,33){$y=\rescaled{\ord\alpha}(x)$}
\put(72,13){$y=f(x)-\varepsilon$}
% Steps
\put(1,60){\line(1,0){17.5}}
\put(0,60){\circle*{3}}
\put(20,60){\circle{3}}

\put(21,50){\line(1,0){17.5}}
\put(20,50){\circle*{3}}
\put(40,50){\circle{3}}

\put(41,40){\line(1,0){7.5}}
\put(40,40){\circle*{3}}
\put(50,40){\circle{3}}

\put(51,35){\line(1,0){17.5}}
\put(50,35){\circle*{3}}
\put(70,35){\circle{3}}

\end{picture}
\end{tabular}
\caption{An example of a composition $\alpha$ and a decreasing function $f$ showing that if $\rescaled{\alpha}(x)$ is enclosed between $f(x)-\varepsilon$ and $f(x)+\varepsilon$, then so is $\rescaled{\ord\alpha}(x)$, an immediate consequence of Lemma~\ref{lem:reordering}.  \label{fig:reordering_proof}}
\end{figure}

Clearly, Lemma~\ref{lem:reordering} holds true also when the max-norm is replaced by the weaker 
convergence condition used in our limit shape definition \eqref{eq:limitshape_def}.
\begin{lemma}
\label{lem:reordering_cor}
For any distribution $\pi_n$ on $\mathcal{W}(n)$, define a corresponding distribution $\tilde{\rho}^{(n)}$ on $\Par(n)$  by
\begin{equation}
\label{eq:reordering_cor}
\tilde{\rho}^{(n)}(\lambda)=\sum_{\substack{ \alpha\in\mathcal{W}(n) \\ \ord\alpha=\lambda }}\pi_n(\alpha).
\end{equation}
If $\phi$ is a limit shape of $\pi_n$ on $\mathcal{W}(n)$ then $\phi $ is also a limit shape of $\tilde{\rho}^{(n)}$ on $\Par(n)$.
\end{lemma}
%Formally, $\nu^{(n)}$ depends on $\pi_n$, but for simplicity of notation, we do not indicate in $\nu^{(n)}$ the dependence on $\pi_n$, but assume the underlying distribution on $\mathcal{W}(n)$. 
\begin{proof}
The assumption that $\phi$ is a limit shape of the distribution $\pi_n$ on $\mathcal{W}(n)$  means that
\[
%\lim_{n\rightarrow\infty}\pi_n\left\lbrace \alpha\in\mathcal{W}(n): \|\rescaled{\alpha}-\phi\|_{\infty}<\varepsilon\right\rbrace=1.
\lim_{n\rightarrow\infty}\pi_n\left\lbrace \alpha\in\mathcal{W}(n): |\rescaled{\alpha}(x)-\phi(x)| < \varepsilon\right\rbrace=1.
\]
for all $x>0$. By virtue of Lemma~\ref{lem:reordering} we can replace $\alpha$ with $\ord\alpha$ in this formula:
\begin{equation}
\label{eq:lemmaproof1}
%\lim_{n\rightarrow\infty}\pi_n\left\lbrace \alpha\in\mathcal{W}(n): \|\rescaled{\ord\alpha}-\phi\|_{\infty}<\varepsilon\right\rbrace=1.
\lim_{n\rightarrow\infty}\pi_n\left\lbrace \alpha\in\mathcal{W}(n): |(\rescaled{\ord\alpha})(x)-\phi(x)| < \varepsilon\right\rbrace=1.
\end{equation}
The set 
%$A:=\left\lbrace \alpha\in\mathcal{W}(n): \|\rescaled{\ord\alpha}-\phi\|_{\infty}<\varepsilon\right\rbrace$
$A:=\left\lbrace \alpha\in\mathcal{W}(n): |(\rescaled{\ord\alpha})(x)-\phi(x)| < \varepsilon
\text{ for all }x>0 \right\rbrace$
can be written as a disjoint union of equivalence classes with respect to sorting:
\[
A=\bigcup_{\lambda\in L} \{ \alpha\in\mathcal{W}(n):\, \ord\alpha=\lambda \}
\]
where $L=\{ \lambda\in\Par(n):\, |\rescaled{\lambda}(x)-\phi(x)| < \varepsilon \text{ for all }x>0 \}$. The $\pi_n$-probability measure of $A$ is
\begin{align*}
\pi_n(A)
& = \pi_n\left( \bigcup_{\lambda\in L} \{ \alpha\in\mathcal{W}(n):\, \ord\alpha=\lambda \} \right) \\
& = \sum_{\lambda\in L} \pi_n\{ \alpha\in\mathcal{W}(n):\, \ord\alpha=\lambda \} \\
& = \sum_{\lambda\in L} \; \sum_{\substack{  \alpha\in\mathcal{W}(n) \\ \ord\alpha=\lambda }} \pi_n(\alpha) \\
& = \sum_{\lambda\in L} \tilde{\rho}^{(n)}(\lambda) && \text{(by \eqref{eq:reordering_cor})} \\
& = \tilde{\rho}^{(n)}(L).
\end{align*}
From \eqref{eq:lemmaproof1} we have that $\lim_{n\rightarrow\infty}\pi_n(A)=1$. Because $\pi_n(A)=\tilde{\rho}^{(n)}(L)$, we can conclude that also ${\lim_{n\rightarrow\infty}\tilde{\rho}^{(n)}(L)=1}$, that is,
\[
  \lim_{n\rightarrow\infty}\tilde{\rho}^{(n)}\left\lbrace \lambda\in\mathcal\Par(n): |\rescaled{\lambda}(x)-\phi(x)|<\varepsilon \text{ for all }x>0 \right\rbrace = 1.
\]
This means that $\phi$ is a limit shape of the distribution $\pi_n$ on $\Par(n)$.
\end{proof}

\section{Three regimes}\label{sec:regimes}
Recall from Section~\ref{sec:qn-Bulgarian} the $q_n$-proportion Bulgarian solitaire developed in \cite{EJS}, where the limit shape is triangular when $q_n^2 n\rightarrow 0$, exponential when $q_n^2 n\rightarrow \infty$
and an interpolation between the two when $q_n^2 n\rightarrow C>0$.

The $p_n$-random $q_n$-proportion Bulgarian solitaire seems to share this property of three regimes of limit shapes. Specifically, in Section~\ref{sec:conj} we conjecture the limit shape to be triangular when $p_n q_n^2 n\to 0$, exponential when $p_n q_n^2 n\to\infty$ and an interpolation between the two (a piecewise linear function graph that depends on $C$) when $p_n q_n^2 n\to C>0$.

The focus in this paper is the exponential regime of the $p_n$-random $q_n$-candidate Bulgarian solitaire,
i.e.\ the case $p_n q_n^2 n\to\infty$ as $n\to\infty$. However, with the proof technique we employ we will prove the stronger statement that the limit shape holds even when the configurations are considered elements of $\mathcal{W}(n)$, i.e.\ even without sorting the piles of a configuration according to size to create a partition in $\mathcal{P}(n)$. We will instead require the stronger condition $p_n q_n^2 n/{\log n}\to\infty$ as $n\to\infty$. By virtue of Lemma~\ref{lem:reordering_cor}, the limit shape will also hold for partitions.

%we identify three regimes in which the impact of the rounding involved in a move of $p$-random $q$-candidate Bulgarian solitaire differ (recall that the number of cards that are candidates to be picked from a pile is rounded upward to the nearest integer). These regimes are based on the asymptotic behavior of $p_n$ and $q_n$. Specifically, en exponential limit shape arises in the regime $\frac{np_n q_n^2}{\log n} \to \infty$. The limit shapes in the two other regimes, $np_nq_n^2 \to 0$ and $np_nq_n^2 \to C$ for some constant $C$, are conjectured in section \ref{sec:conj}.

\section{The exponential limit shape}\label{sec:exp-limit}
Here we investigate the limit shape of configurations in the $p_n$-random $q_n$-candidate
Bulgarian solitaire $\mathscr{B}(n,p_n,q_n)$ in the regime
\begin{equation}
\label{eq:npq2/logn}
\frac{p_n q_n^2 n}{\log n}\to\infty \text{ as }n\to\infty.
\end{equation}
Our main result, Theorem~\ref{thm:main_new}, says that, under the additional asymptotic property $p_n q_n \to 0$ as $n\rightarrow\infty$, the boundary function of the diagram, downscaled, will
%eventually (after playing a sufficiently large number of moves) 
resemble the exponential shape $e^{-x}$ asymptotically almost surely (a.a.s.), i.e.\ with a probability that tends to 1 as $n\to\infty$. See Figure~\ref{fig:sim}. Throughout this section, ``a.a.s.'' can be read as ``with a probability that tends to 1 as $n\to\infty$''. Also, the asymptotic notations $o$ and $O$ will always be with respect to $n\to\infty$.

\begin{figure}[ht]
	\setlength{\unitlength}{0.08cm}
	\centering
	\begin{picture}(153,70)	
	% l b r t
	\put(0,0){\includegraphics[scale=0.7,clip=true,trim=60 285 40 335]{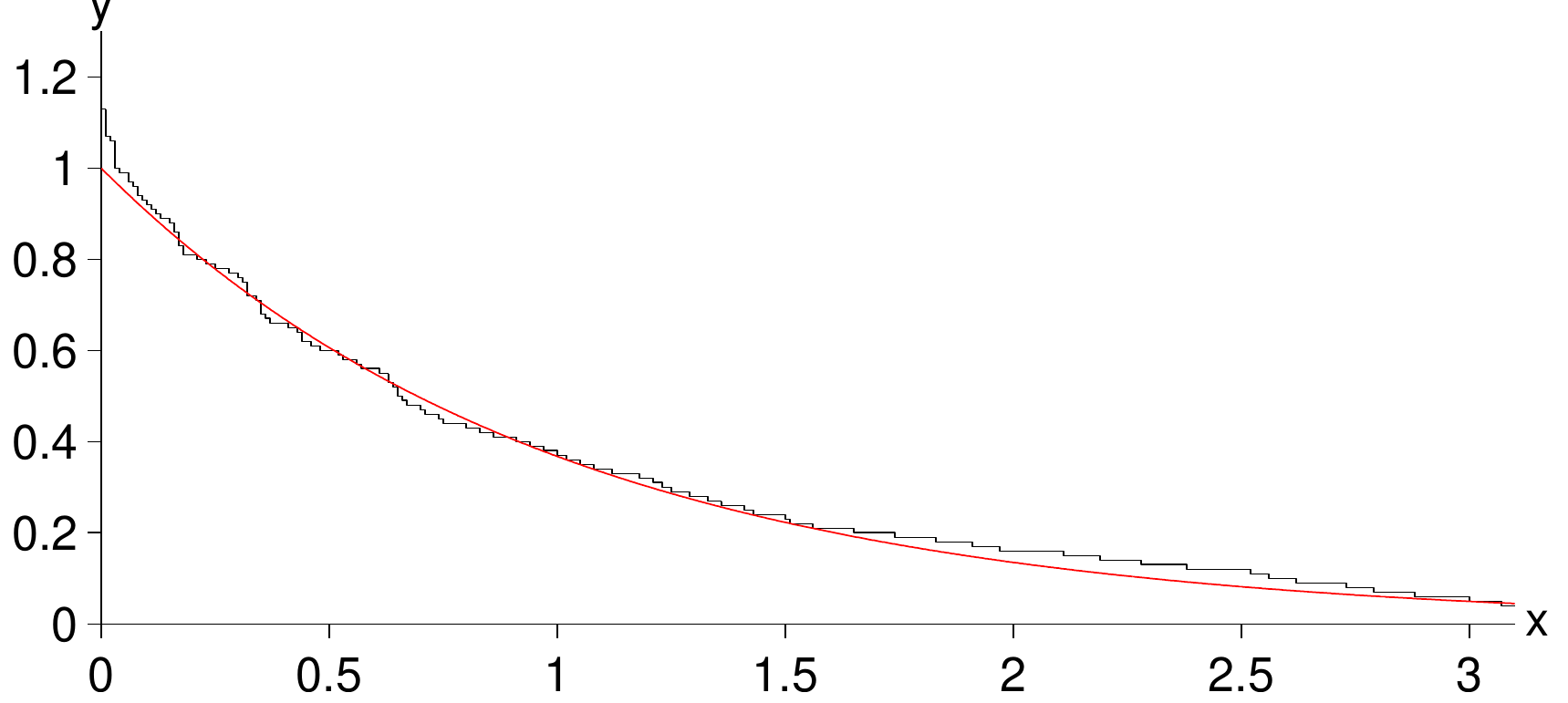}}
	\end{picture}
	\caption{The result of a computer simulation after 200 moves of $p_n$-random $q_n$-proportion Bulgarian solitaire in the case $q_n=1$, with $n=10^5$ cards and $p_n=0.01$, starting from a triangular configuration. The jagged curve is the rescaled diagram-boundary function of the resulting configuration and the smooth curve is the limit shape $y=e^{-x}$. \label{fig:sim}}
\end{figure}

We shall see that the condition $p_n q_n^2 n/{\log n}\to\infty$ implies that the rounding effect in computing the number of candidate cards is negligible. Thus, the number of candidate cards will tend to $q_nn$ as $n\to\infty$. This in turn means that the expected number of picked cards will eventually be close to $p_n q_n n$, thus $\lambda_1\approx p_n q_n n$ is the size of the pile created in a move of the solitaire. Recall from Section \ref{sc:limit-concept} that the scaling factor we employ is $a_n = n/\lambda_1 = \frac{1}{p_nq_n}$. Thus, if $p_n q_n$ is bounded away from zero, then the scaling $\frac{1}{p_nq_n}$ is bounded and hence cannot transform the jumpy boundary diagrams into a smooth limit shape. Therefore, we also require
\begin{equation}
\label{eq:pq}
p_n q_n \rightarrow 0 \text{ as }n\rightarrow\infty.
\end{equation}
On the other hand, if $p_nq_n$ tends to zero too fast, the pile sizes will be small and their random fluctuations will be large. For instance, the new pile after each move has a size drawn from the binomial distribution $\Bin(K,p_n)$, where $K\approx q_nn$ is the number of candidate cards, with relative standard deviation $\sim 1/\sqrt{p_n q_n n}$. The requirement \eqref{eq:npq2/logn} guarantees that $p_nq_n$ does not tend to zero too fast.

\begin{theorem}
	\label{thm:main_new}
	For each positive integer $n$, pick $q_n$ and $p_n$ with $0<p_n,q_n\le 1$ and a (possibly random)
	initial configuration $\lambda^{(0)}\in\Par(n)$. Let $(\lambda^{(0)},\lambda^{(1)},\dotsc)$ be the Markov chain on $\Par(n)$ defined by $\mathscr{B}(n,p_n,q_n)$, and denote its stationary measure by $\pi_{n,p_n,q_n}$. Suppose
	\[
	p_n q_n\rightarrow 0 \quad\text{and}\quad \frac{p_n q_n^2 n}{\log n} \rightarrow\infty \quad\text{as } n\rightarrow\infty.
	\]
	%Then, for any $m>D+M$ where $D=c\frac{\log n}{p_nq_n}$ and $M=\lceil n^2/p_n \rceil$,
	%the probability distribution for the resulting diagram $\lambda^{(m)}$ after playing $m$ moves has the limit shape $e^{-x}$ under the scaling $a_n=\frac{1}{p_n q_n}$.
	Then $\pi_{n,p_n,q_n}$ has the limit shape $e^{-x}$ under the scaling $a_n=(p_n q_n)^{-1}$.
\end{theorem}

%\begin{theorem}
%	\label{thm:main_new}
%	For each positive integer $n$, pick a fraction $q_n\in(0,1]$, a probability $p_n\in(0,1)$, a (possibly random)
%	initial configuration $\lambda^{(0)}\in\Par(n)$ and let $(\lambda^{(0)},\lambda^{(1)},\dotsc)$ be the Markov chain on $\Par(n)$ defined by $\mathscr{B}(n,p_n,q_n)$.
%	
%	Suppose that $p_n$ and $q_n$ has the asymptotic properties
%	\[
%	p_n q_n\rightarrow 0 \quad\text{and}\quad \frac{np_n q_n^2}{\log n} \rightarrow\infty \quad\text{as } n\rightarrow\infty.
%	\]
%	Then, for any $m>D+M$ where $D=c\frac{\log n}{p_nq_n}$ and $M=\lceil n^2/p_n \rceil$,
%	%$m=m(\varepsilon_1,n)>\frac{\varepsilon^2}{2+\varepsilon}n$,
%	the probability distribution for the resulting diagram $\lambda^{(m)}$ after playing $m$ moves has the limit shape $e^{-x}$ under the scaling 
%	%$a_n=\frac{1}{p_n q_n}$
%	$a_n=(p_n q_n)^{-1}$.
%	%In fact, for any $\varepsilon>0$ we have
%	%\[
%	%P\left(\| \rescaled{\lambda^{(m)}}-g\|_{\infty} \leq \varepsilon \right)
%	%\geq 1-\exp\left[...\right].
%	%\]
%\end{theorem}

The proof of Theorem~\ref{thm:main_new} heavily relies on the following version of Chernoff bounds.
%, which we state as a proposition for the sake of self-containment.
For a proof, see for example \cite{mitzenmacher2005probability}.
\begin{chernoff}
	For $n\ge 1$ and $0<p\le 1$, let $X\sim\Bin(n,p)$ and set $\mu=E(X)=np$. Then, for any $0<\gamma<\mu$,
	\begin{equation}
	\label{eq:chernoff}
	P(\left|X-\mu\right|\geq \gamma) \leq 2\exp\left(-\frac{\gamma^2}{3\mu}\right).
	\end{equation}
\end{chernoff}
%\begin{proof}
%	Recall the standard Chernoff bounds $P(X\ge (1+\delta)\mu) \le \exp(-\delta^2\mu/3)$ and $P(X\le (1-\delta)\mu) \le \exp(-\delta^2\mu/2)$ for $0<\delta<1$. The claim follows immediately by setting $\gamma = \delta\mu$.
%\end{proof}
The idea of the proof of Theorem~\ref{thm:main_new} is the following. 

We will use the approach developed in Section~\ref{sec:repr_compositions}, i.e.\ card configurations in the solitaire will be represented by weak integer compositions and the piles are ordered with respect to creation time, i.e.\ if $\alpha\in\mathcal{W}(n)$ is the current configuration in the solitaire, then $\alpha_1$ was the last formed pile, $\alpha_2$ the pile that was formed two moves ago, etc. With this representation, some piles may be empty, so one may imagine each pile being placed in a bowl and the bowls are lined up in a row on the table. In each move of the solitaire, the new (possibly empty) pile is put in a new bowl to the left of all old bowls. As mentioned in Section~\ref{sec:regimes}, we shall prove Theorem~\ref{thm:main_new} as a limit shape result for diagram-boundary functions of \emph{compositions}. Thus, throughout this section, each configuration of $n$ cards will be represented by an element of $\mathcal{W}(n)$. Also, in the following we may abbreviate $p=p_n$ and $q=q_n$ unless the dependency on $n$ is crucial.

Assume a configuration $\alpha = (\alpha_1,\alpha_2,\dotsc,\alpha_N,0,0,\dotsc)$ of $n$ cards with $N=N(\alpha)$ piles (so that $\sum_{i=1}^N \alpha_i=n$) in the solitaire $\mathscr{B}(n,p,q)$. The number of candidate cards in the next
move is $\kappa := \sum_{i=1}^N \lceil q\alpha_i\rceil$.
%Since the number of candidate cards in a move is the sum of fractions of the pile sizes rounded upwards to the nearest integer,
%the number of candidate cards is $nq+R$ where $R$ is the sum of all round-off effects,
%\marginpar{Only if round-off effect is negligible? Circular reasoning}
We denote the rounding effect in pile $1\le i\le N$ by
$R_i := \lceil q\alpha_i\rceil-q\alpha_i$ and the total rounding effect by
$R := \kappa-qn = \sum_{i=1}^N R_i$.

Clearly, $R<N$ (since $R_i<1$ for any $i$), i.e.\ the total rounding effect is bounded above by the number of piles. The first thing we will do is to make sure that after a sufficient number $D$ of moves from the initial configuration $\alpha^{(0)}$, 
%$N=o(qn)$ so that $R=o(qn)$ and thus $\kappa = qn+R = qn+o(qn)$.
the number of piles $N(\alpha^{(D)})$ in the resulting configuration $\alpha^{(D)}$ is much smaller than $qn$ a.a.s.\ (so that the number of candidate cards $\kappa$ is approximately $qn$ and thus the total rounding effect $R$ is negligible). 
In Lemma~\ref{lem:D} we show that it is possible to choose such a $D$, namely $D=c\frac{\log n}{pq}$ for any $c\ge 14$.
%, and in fact $N(\alpha^{(D)})=o(qn)$.

We also need to make sure that the number of piles stays $o(qn)$ for sufficiently many additional moves $M$, long enough
to establish the convergence of the overall shape. Lemma~\ref{lem:D} will also guarantee that $M=\lceil n^2/p \rceil$ suffices for this purpose.

Thus, in the following we shall use
\begin{equation}
\label{eq:DM}
D = \left\lceil 14\frac{\log n}{pq} \right\rceil \quad \text{and} \quad M = \left\lceil \frac{n^2}{p} \right\rceil.
\end{equation}	

If the number of piles stays $o(qn)$ during $M$ moves so that the number of candidate cards stays approximately $qn$, the newly formed pile in each of these moves will have expected size $pqn$. Our proof technique involves studying the evolution of such a pile (which will follow an exponential decay in size). Therefore we need to additionally make sure that no old piles (which could potentially be much larger than $pqn$) remain after these $M$ moves. Lemma~\ref{lem:M} shows that, in fact, after $M$ moves all piles in the starting configuration have disappeared a.a.s.
%starting from any configuration $\alpha$, 
%$\alpha$

%Note that after the $D$ moves specified in Lemma~\ref{lem:D}, we can say something about the \emph{number of piles}. The next step is to control the sizes of these.
%To this end, in Lemma~\ref{lem:M} we will show that it is possible to choose $M$ in such a way that both (i) after an additional $M$ moves, \emph{all} piles (including such large piles in $\alpha^{(D)}$) have disappeared a.a.s.; and at the same time (ii) during all these $M$ moves, the number of piles stays $o(pn)$.

\begin{lemma}
	\label{lem:M}
	Let $M$ be given by \eqref{eq:DM}. From any initial configuration $\alpha\in\mathcal{W}(n)$, after $M$ moves in the solitaire $\mathscr{B}(n,p_n,q_n)$,	all piles in $\alpha$ have been consumed a.a.s.
	%A pile of size $n$ (worst case) disappears (with high probability) after $M=\lceil n^2/p \rceil$ moves
	%, and the piles formed by these $M$ moves are small
\end{lemma}
\begin{proof}
	%We prove this by showing that even a pile of size $n$ in $\alpha$ will disappear a.a.s.\ after $M$ moves.
	%This means that a pile of smaller size will also disappear after the same number of moves. 
	Consider a pile of size $n$. The size of this pile after $M$ moves is statistically dominated by $\max(n-X,0)$ where $X\sim\Bin(M,p_n)$ whose expected value is $E(X)=Mp_n=\lceil n^2/p_n \rceil p_n > n$. Therefore, the probability that the pile remains after $M$ moves is ${P(X<n)}$ with the bound
	\begin{align*}
	P(X<n) &\le P(|X-Mp_n|>|Mp_n-n|) \le 2\exp\left( -\frac{(Mp_n-n)^2}{3Mp_n} \right) \\
	&\le 2\exp\left(-\frac{n^2}{3}(1+o(1))\right),
	\end{align*}
	where we used the Chernoff bound \eqref{eq:chernoff}.
	Thus, since any given pile in $\alpha$ has size $\le n$, and the number of piles (of any size in any configuration) is $\le n$, the probability that \emph{all} piles in $\alpha$
	%that existed after $D$ moves 
	have been consumed after $M$ moves is at least
	\[
	1-2n \exp\left(-\frac{n^2}{3}(1+o(1))\right) \to 1
	\]
	which concludes the proof.
\end{proof}

\begin{lemma}
\label{lem:D}
Let $n,p_n,q_n$ and an initial configuration $\alpha^{(0)}$ be given in the solitaire $\mathscr{B}(n,p_n,q_n)$. Then
\[
\frac{1}{q_n n} \max \left\{ N(\alpha^{(D+1)}), \dotsc, N(\alpha^{(D+M)}) \right\} \to 0 \text{ in probability},
\]
where $D$ and $M$ are given by \eqref{eq:DM}.
%(i) There is a constant $c \ge 14$ such that, after $D=\lceil c\frac{\log n}{pq} \rceil$ moves
%the number of piles $N(\alpha^{(D)})$ in the resulting configuration is $o(qn)$.
%
%(ii) In each of the $M=\lceil n^2/p \rceil$ configurations 
%$\alpha^{(D+1)}, \alpha^{(D+2)}, \dotsc, \alpha^{(D+M)}$, the number of piles is $o(qn)$.
\end{lemma}
\begin{proof}
Let us abbreviate $p=p_n$ and $q=q_n$.
We will first prove that all piles of size at most $q^{-1}\log n$ disappear with high probability after $D$ moves, making sure that there are not many small piles in $\alpha^{(D)}$.
%The observation we will use is this: Every nonempty pile decreases by at least $1$ with probability at least $p$ in each move.
Consider a pile of size at most $q^{-1}\log n$ in $\alpha^{(0)}$.
Note that every nonempty pile decreases by at least $1$ with probability at least $p$ in each move.
Therefore, after $D$ moves the number of picked cards from this pile statistically dominates $X\sim\Bin(D,p)$ with expected value $Dp=14q^{-1}\log n$.
%Thus, the size of this pile is statistically dominated by $\frac{\log n}{q} - X$. 
Using the Chernoff bound \eqref{eq:chernoff}, the probability that this pile remains after $D$ moves is at most
\begin{align*}
P_1 &:= P\left( X<\frac{\log n}{q} \right) \le P\left( |X-Dp|> \left| Dp-\frac{\log n}{q} \right| \right) \\
& \le 2\exp\left(-\frac{\Bigl(Dp-\frac{\log n}{q}\Bigr)^2}{3Dp}\right)
= 2\exp\left( -\frac{( \frac{14\log n}{q} - \frac{\log n}{q} )^2}{3\cdot 14\frac{\log n}{q}} \right)
= 2n^{-\frac{13^2}{42}\frac{1}{q}} < 2n^{-4}.
% &= \exp(-(1/q)\log n \cdot \Theta(1)).
\end{align*}
Since there can be at most $n$ piles of size at most $q^{-1}\log n$, the probability that not \emph{all} piles of size at most $q^{-1}\log n$ have disappeared after $D$ moves is bounded by
\[
P_2:=nP_1=2n^{-3}.
\]
%which obviously tends to zero as $n\to\infty$.

Let us now turn our attention to \emph{the number of} piles after these $D$ moves.
By the above, all piles smaller than $q^{-1}\log n$ have disappeared with high probability. Clearly, the number of piles \emph{larger than} $q^{-1}\log n$ can never be more than $\frac{n}{q^{-1}\log n}=\frac{qn}{\log n}$. Also, during the process of these $D$ moves, at most $D$ new piles have been formed. (Exactly $D$ piles have been formed but some may have disappeared in the process.) Thus, for the total number of piles $N(\alpha^{(D)})$ in the configuration $\alpha^{(D)}$ after $D$ moves, with probability at least $1-2n^{-3}$, we have
\[
N(\alpha^{(D)}) \le \frac{qn}{\log n}+D=\frac{qn}{\log n}+14\frac{\log n}{pq}
= qn\left( \frac{1}{\log n}+14\frac{\log n}{p q^2 n} \right) = o(qn),
\]
where we used the assumption \eqref{eq:npq2/logn} in the last step. 
It follows that, for any $\varepsilon>0$,
\[
\frac{1}{qn} \max \left\{ N(\alpha^{(D+1)}), \dotsc, N(\alpha^{(D+M)}) \right\} < \varepsilon
\]
with probability at least $1-2n^{-3}M \ge 1-\frac{2}{pn} \to 1$ since $pn\to \infty$. (That $pn\to 0$ is also a consequence of the assumption \eqref{eq:npq2/logn}.)
%we have
%\[
%N(\alpha^{(D)}) \le \frac{nq}{\log n}+D=\frac{nq}{\log n}+c\frac{\log n}{pq}=o(qn),
%\]
%where we used the assumption \eqref{eq:npq2/logn} in the last step. 
%
%This concludes the proof of (i).
%
%(ii) 
%%(This is good since every pile can contribute with a rounding off error of maximum 1 in computing the number of candidate cards. So in worst case we get [number of piles] more candidate cards than $nq$. We know that the number of candidate cards is approximately $nq$, so we want the number of piles to be $<<nq$ because then the "round of error" is then negligible.)
%%
%%So after $D$ moves we have few piles. BUT some of these may be big. After an additional $M$ moves even the 
%%worst case (pile size $n$) has disappeared with high probability. And not only that, but ALL piles of the worst kind have
%%disappeared, since even $n$ times this probability $\to 0$.
%%
%During the $M$ moves from $\alpha^{(D)}$ to $\alpha^{(D+M)}$, the probability that any small ($<\frac{\log n}{q}$) pile has been formed is
%$\le M\cdot P_2 < (n^2/p+1) \cdot 2n^{-3} = \frac{2}{pn} + 2n^{-3} \to 0$ as $n\to\infty$ (since $pn\to\infty$). Thus, by considering the last $D$ moves in each of the configurations
%$\alpha^{(D+1)}, \alpha^{(D+2)}, \dotsc, \alpha^{(D+M)}$, the number of piles continues to be
%of size $o(qn)$ during these $M$ moves.
\end{proof}

Lemma~\ref{lem:D} asserts that the number of piles remains to be $o(qn)$ during the $M$ moves from $\alpha^{(D)}$ to $\alpha^{(D+M)}$, hence the number of candidate cards remains to be $qn$ (a.a.s.) during the same moves. Therefore the number of picked cards (which equals the size of the newly formed pile), remains of \emph{expected} size $pqn$.
In Lemma~\ref{lem:variance} we prove that the \emph{actual} number of picked cards in each of these $M$ moves does not deviate (relatively) from $pqn$.
%with more than $\varepsilon$ a.a.s.

%\paragraph{Thus, during these $M$ moves the newly formed pile will have size $\lambda_1 \sim npq$} So after $D+M$ moves, all original piles (after the first $D$ moves) have disappeared. During this whole $M$ process
%the newly formed pile has always been of size $npq$, since the number of piles has always been bounded by $nq$.
%
%After $D$ moves (or any power of $n$ moves) 
%
%Thus, after $D+M$ moves, all piles that were present after $D$ moves are consumed (a.a.s.) and all current piles are ``fresh'' in the sense that they were formed during these $M$ moves (a.a.s.). By Lemma~\ref{lem:M}, during these $M$ moves, the number of \emph{candidate} cards has consistently been $nq$, and thus the number of \emph{picked} cards has been on average $npq$ in each move.
%
%Even if the number of candidate cards is $\sim nq$, the number of \emph{actually} picked cards could still vary due to randomness. For example, what if we pick ALL candidate cards in a move creating a huge pile of size $nq$? We must limit the risk of this by investigating the variance of the size $\lambda_1$ of the new pile:
\begin{lemma}
	\label{lem:variance}
	Let $n,p_n,q_n$ and an initial configuration $\alpha^{(0)}$ be given in the solitaire $\mathscr{B}(n,p_n,q_n)$.
	Let $D$ and $M$ be given by \eqref{eq:DM}. Then
	\[
	\max_{k\in[D+1,D+M]} \frac{|\alpha^{(k)}_1 - p_n q_n n|}{p_n q_n n} \to 0 \text{ in probability}
	\]	
	as $n\to\infty$.
	%, for each $k=D+1,\dotsc,D+M$, the newly formed pile $\alpha_1^{(k)}$ satisfies
	%Starting from any configuration, for any integer $m>nq$, after $m$ moves, the
	%size $\lambda_1$ of the newly formed pile satisfies	
	%\[
	%P\left(|\alpha_1^{(k)}-npq| >  \varepsilon npq\right) = o(1)
	%\]
	%for any $\varepsilon>0$.
\end{lemma}
\begin{proof}
%During the $M$ moves
Let us abbreviate $p=p_n$ and $q=q_n$.
Let $\varepsilon>0$ and let $\kappa$ be the number of candidate cards in $\alpha^{(k-1)}$ for some $k=D+1,\dotsc,D+M$.
Recall that the total rounding effect in computing the number of candidate cards is bounded above by the number of piles.
It therefore follows from Lemma~\ref{lem:D} that $\kappa=nq(1+o(1))$.
The new pile size is $\alpha_1^{(k)} \sim \Bin(\kappa,p)$. Then, using the triangle inequality and the Chernoff bound \eqref{eq:chernoff} we have
\begin{align*}
P_3 := P(|\alpha^{(k)}_1 -pqn| > \varepsilon pqn) &\le P(|\alpha^{(k)}_1-\kappa p| > \varepsilon pqn-|\kappa p-pqn|) \\
&< 2\exp\left( -\frac{(\varepsilon pqn-|\kappa p-pqn|)^2}{3\kappa p} \right) \\
&= 2\exp\left( -\frac{\varepsilon^2}{3} pqn(1+o(1)) \right) \\
&= o(1/M)
\end{align*}
%(WE WILL USE THIS $1/M$ (NUMBER OF PILES) TIMES)\newline
where the last equality is derived as follows.
%By \eqref{eq:npq2/logn} we have $\frac{1}{pq}=o(n)$ which implies $\log\frac{1}{pq}=o(\log n)$. By \eqref{eq:npq2/logn} we also have $\log n = o(npq)$. Therefore, $\log\frac{1}{pq} = o(npq)$, and thus $\exp(-npq)=o(pq)$.
By \eqref{eq:npq2/logn}, $\log n = o(pqn)$ and hence $\log n^a = o(pqn)$ for any $a\ge 1$. 
Since $pqn\to\infty$, this means that $\exp(-pqn)$ tends to zero faster than $\exp(-\log n^a)$, i.e.,
%switching signs and exponentiating yields
$\exp(-pqn) = o(1/n^a)$. Since $np\to\infty$, we therefore also have $\exp(-pqn) = o(p/n^a) = o(1/M)$.
The next to the last equality follows from the fact that $\varepsilon pqn$ dominates over $|pqn-\kappa p|$
(since $|pqn-\kappa p| = |pqn-nq(1+o(1))p| = pqn\cdot o(1)$).
%.\marginpar{Why?}

Therefore, the probability that $|\alpha_1^{(k)} - pqn| > \varepsilon pqn$ for \emph{any} $k$ during the entire process of $M$ moves is bounded by $MP_3 = M\cdot o(1/M) = o(1)$.
%In order to be able to say that the newly formed pile is $npq$. During the $M$ moves it takes to get rid of the large piles, we must know that among the $M$ newly created piles, there is no pile which is still big. So the probability above must also be multiplied by $M$.	
\end{proof}
While playing the solitaire, there is a
possibility that at some point there will be too many piles, and thereby 
the number of candidate cards will be bigger than $qn$ (and thus the size of the newly formed pile will be bigger than $pqn$). Lemmas~\ref{lem:M} and \ref{lem:D} ensures that this never happens a.a.s.\ during
%risk is small enough
the entire process of $M$ moves from $\alpha^{(D)}$ to $\alpha^{(D+M)}$.

There is also a risk that, even if there are suitably many ($qn$) candidate cards, the number of \emph{picked} cards among them will deviate from $pqn$ due to random fluctuations (and thereby the size of the newly formed pile will deviate from $pqn$). Lemma~\ref{lem:variance} ensures that this never happens a.a.s.\ during the same period of $M$ moves.

Therefore, after $m:=D+M$ moves we have the following a.a.s.
\begin{itemize}
\item all current piles have been formed during the last $M$ moves, and
\item all current piles had size $pqn$ when they were formed.
\end{itemize}
At this point, i.e.\ in the configuration $\Gamma:=\alpha^{(m)}$, the leftmost pile (of size $\Gamma_1$) was formed one move ago, the second pile from the left (of size $\Gamma_2$) was formed two moves ago, and so on. We shall prove that the size $\Gamma_k$ of the pile that was formed $k$ moves ago for any $k=1,2,\dotsc,m$ is
$\Gamma_k = \Gamma_1(1-pq)^k = pqn(1-pq)^k$ a.a.s., i.e.\ the size decreases exponentially with $k$ with decay factor $1-pq$.
%yielding an exponential limit shape.

We will now consider the evolution of a given pile of size $A_1$ during $r\ge 1$ steps in the $p$-random $q$-proportion Bulgarian solitaire in the following way. We will need to keep track of each individual card in this pile. To this end, we
label the cards $1,2,\dotsc,A_1$ starting from the top, and each card will keep their label throughout the process.
Let $X_{i,k} \in \{0,1\}$ where $i=1,\dotsc,A_1$ and $k=1,\dotsc,r$ be independent Bernoulli random variables with $P(X_{i,k}=1) = p$.

%The $q_n$-fractional $p_n$-random solitaire
Consider the following process. 
Let $A_{k+1}$ be the number of cards after $k$ moves.
In each move $k=1,2,\dots,r$, we remove the card with label $i$ if $X_{i,k}=1$ and this card belongs to the candidate cards, i.e., the $\lceil qA_k \rceil$ top-most remaining cards. 
We will call this process a $q$-\emph{process}. This process describes the evolution of a pile of size $A_1$ in the $p$-random $q$-proportion Bulgarian solitaire. 

Using the same Bernoulli variables, for any real number $0 \le s \le 1$, we define an $s$-\emph{threshold process} in the following way. 
In each move $k=1,2\dots,r$, we remove the card with label $i$ if $X_{i,k}=1$ and $i\le sA_1$. In this process, we let $A^{[s]}_{k+1}$ denote the number of remaining cards after $k$ moves. 
When it is relevant to indicate the initial pile size, an $s$-threshold process is called
an $(s,A_1)$-threshold process and the number of remaining cards after $k$ moves is denoted by $A^{[s,A_1]}_{k+1}$.

In the proof of Theorem~\ref{thm:main_new}, we will use two different $s$-threshold processes (for two different values of $s$) to over- and underestimate the sizes of $r+1$ consecutive piles in $\Gamma$ (corresponding to the $r$ steps in an $s$-threshold process). Both these processes will have the same desired limit shape and thus the limit shape of our solitaire will follow by the squeeze theorem. We first need a combinatorial lemma giving sufficient conditions for overestimation and for underestimation.
%under which the pile sizes in the $s$-threshold process overestimate and underestimate the corresponding pile sizes in the $q$-process.
%that will be used to show when these two processes actually over- and underestimate the pile sizes in the solitaire.
\begin{lemma}
\label{lem:combinatorial}
	(i) If $\lceil sA_1 \rceil \le \lceil qA_1 \rceil$, then $A^{[s]}_k \ge A_k$ for $k=1,\dotsc,r+1$.
	%(ii) If $(1-q)A_k \ge A_0-m$ for any $k$, then $A^{(m)}_k \le A_k$ for any $k$.
	
	(ii) If $(1-q)A^{[s]}_{r+1} \ge A_1-\lceil sA_1 \rceil$, then $A^{[s]}_k \le A_k$ for $k=1,\dotsc,r+1$.
\end{lemma}
\begin{proof}
	(i) A card that is removed at some step $\ell$ during the $s$-threshold process must have label $i\le\lceil sA_1\rceil$,
	so in the $q$-process it belongs to the $\lceil qA_1\rceil$ candidate cards in the initial pile and hence it belongs
	to the candidate cards also at step $\ell$ and will be removed. Thus, every card removed in the $s$-threshold process
	is removed in the $q$-process too, and it follows that $A_k^{[s]}\ge A_k$ for $k=1,\dotsc,r+1$.
	
	(ii) We show by induction over $r$ that, after $r$ steps, the remaining cards in the $s$-threshold process
	is a subset of the remaining cards in the $q$-process. Suppose $(1-q) A_{r+1}^{[s]}\ge A_1-\lceil sA_1\rceil$. Since
	$A_r^{[s]}\ge A_{r+1}^{[s]}$ we have $(1-q) A_r^{[s]}\ge A_1-\lceil sA_1\rceil$ which by the induction
	hypothesis implies that $A_k^{[s]}\le A_k$ for $1\le k\le r$. It follows that
	$(1-q) A_r\ge A_1-\lceil sA_1\rceil$ which in turn implies that
	$A_r-(A_1-\lceil sA_1\rceil) \ge \lceil qA_r\rceil$. This latter inequality means that the $\lceil qA_r\rceil$
	topmost cards before step $r$ in the $q$-process all have labels larger than $A_1-\lceil sA_1\rceil$. Thus,
	if a card is removed in step $r$ in the $q$-process it is also removed in step $r$ or earlier in the $s$-threshold process.
	This concludes the induction step. The base step $r=0$ is trivial.
%	(i) Clearly, $A_k \ge A_{k+1}$ for any $k$. Therefore, in any move $k$, if $\lceil sA_0\rceil \le \lceil qA_0 \rceil$ the number
%	of picked cards in the $s$-threshold process is always less than or equal to the number of picked cards in the solitaire, in other words $A^{[s]}_k \ge A_k$ for any $k$.
%	
%	(ii) Since $A_k \ge A_{k+1}$ for any $k$, it suffices to prove this for $k=r$, which we prove by contradiction. Suppose $(1-q)A_r^{[s]} \ge A_0-\lceil sA_0 \rceil$, and $A^{[s]}_r>A_r$.
%	The number of cards picked in the $s$-threshold process in the $r$th move is
%	\[
%	A_r^{[s]}-(A_0-\lceil sA_0 \rceil) \ge A^{[s]}_r - (1-q)A^{[s]}_r = qA^{[s]}_r > qA_r.
%	\]
%	But $qA_r$ is the number of cards picked in the solitaire, so if (weakly) more cards are picked in the
%	$s$-threshold process than in the solitaire, there are (weakly) fewer cards left. In other words
%	we must have $A^{[s]}_r \le A_r$ which contradicts the assumption $A^{[s]}_r>A_r$.
\end{proof}

Recall that we are considering the configuration $\Gamma = \alpha^{(m)}$ after $m=M+D$ moves in the solitaire from the initial configuration $\alpha^{(0)}$. 
%Lemma~\ref{lem:M} asserts that the number of piles in $\alpha^{(D+M)}$ is less than $M$. 
We will compare the sizes of $r+1$ consecutive piles in $\Gamma$ to the $r+1$ pile sizes in an $s$-threshold process. In order to  make the comparison for \emph{all} piles in $\Gamma$, this will be done for $r+1$ consecutive piles  (which we will call an $r$-\emph{chunk}) at a time. In each $r$-chunk the initial pile size is the corresponding pile size in the solitaire. 
In other words,
$\Gamma_1, \Gamma_2, \dotsc, \Gamma_{r+1}$ will be compared to %$A^{[s]}_0,A^{[s]}_1,\dotsc,A^{[s]}_r$ 
the pile sizes in an $(s,\Gamma_1)$-threshold process (with initial pile size $\Gamma_1$);
and $\Gamma_{r+2}, \Gamma_{r+3}, \dotsc, \Gamma_{2(r+1)}$ will be compared to %$A^{[s]}_0,A^{[s]}_1,\dotsc,A^{[s]}_r$
the pile sizes in an $(s,\Gamma_{r+2})$-threshold process (with initial pile size $\Gamma_{r+2}$), and so on. Let us call the resulting union of $s$-threshold processes an $(r,s)$-\emph{union} process. Thus, if we denote the pile sizes in this $(r,s)$-union process by ${U}_1,{U}_2,\dotsc$, we have
%or simply ${U}_0,{U}_1,\dotsc$ when $s$ is understood.
%where the indexing is chosen so as to match that of parts in $\Gamma=(\Gamma_1,\Gamma_2,\dotsc)$:
\begin{align*}
& U_1 = \Gamma_1 = A_1^{[s,\Gamma_1]}, && U_2 = A_2^{[s,\Gamma_1]}, && \dotsc, && U_{r+1} = A_{r+1}^{[s,\Gamma_1]}, \\
& U_{r+2} = \Gamma_{r+2}=A_1^{[s,\Gamma_{r+2}]}, && U_{r+3} = A_2^{[s,\Gamma_{r+2}]}, && \dotsc, && U_{2(r+1)} = A_{r+1}^{[s,\Gamma_{r+2}]}, \dotsc.
\end{align*}

We intend to use the $(r,s)$-union process to estimate the pile sizes in $\Gamma$.
In an $s$-threshold process, starting with a pile of size $A_1$, the number of remaining cards $B$ above the level $A_1(1-s)$ after $r$ moves is binomially distributed: $B \sim \Bin(A_1 s,(1-p)^r)$.
See Figure~\ref{fig:exp}.
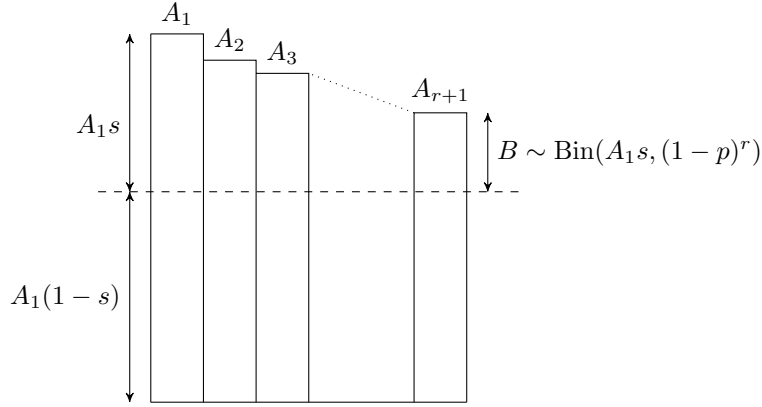
\begin{figure}[ht]
\centering
%\resizebox{6cm}{!}{
\begin{tikzpicture}[x=0.7cm,y=0.7cm,>=stealth']
% horizontal axis
\draw[-] (0,0) -- (6,0);
% vertical axis
\draw[-] (0,0) -- (0,7);

% labels
%\draw (0,0) node[anchor=north] {0}
%      (1,0) node[anchor=north] {1}
%      (2,0) node[anchor=north] {2}
%      (4,0) node[anchor=north] {$\dotsc$}
%      (5,0) node[anchor=north] {$r$};

% Columns
\draw (1,0) -- (1,7) -- (0,7)
      (2,0) -- (2,6.5) -- (1,6.5)
      (3,0) -- (3,6.25) -- (2,6.25)
      (6,0) -- (6,5.5) -- (5,5.5) -- (5,0);
\draw[dotted] (5,5.5) -- (3,6.25);

% s-line
\draw[dashed] (-1,4) -- (7,4);

% Measure lines
\draw[<->] (-0.4,0) -- (-0.4,4);
\draw[<->] (-0.4,4) -- (-0.4,7);
\draw[<->] (6.4,4) -- (6.4,5.5);

% Symbols
\draw (-0.4,2) node[anchor=east]{$A_1(1-s)$};
\draw (-0.4,5.25) node[anchor=east]{$A_1s$};
\draw (6.4,4.75) node[anchor=west]{$B\sim\Bin(A_1s,(1-p)^r)$};

\draw (0.5,7) node[anchor=south]{$A_1$};
\draw (1.5,6.5) node[anchor=south]{$A_2$};
\draw (2.5,6.25) node[anchor=south]{$A_3$};
\draw (5.5,5.5) node[anchor=south]{$A_{r+1}$};

\end{tikzpicture}
%}
\caption{The $r$ steps of an $(s,A_1)$-threshold process.\label{fig:exp}}
\end{figure}
Therefore we need to choose $r=r_n$ and $s=s_n$ in such a way that we have the following in each $s$-threshold process:
\begin{itemize}
	\item[I] The pile size $A_{k+1}$ is close to $A_1(1-pq)^k$ a.a.s.\ for all $k=1,\dotsc,r$, which we need to establish the wanted limit shape. 
	\item[II] At the same time $s$ must be close enough to $q$ to make the over- and underestimations tight enough.
\end{itemize}

%(Lemma~\ref{lem:M} asserts that the number of piles in $\Gamma=\alpha^{(D+M)}$ is less than $M$.) 

To accomplish (I), clearly $r=r_n$ can at least not be chosen bigger than $1/p_n$, in fact we shall require $p_n r_n\to 0$ as $n\to\infty$, in order for the variance in the size of the last pile (after $r_n$ moves) in an $s$-threshold process to be small with high probability. However, we shall see that $p_n r_n$ may not tend to zero too fast. We will require
%We shall now see that we can choose $r=r_n$ large enough that the variance in the 
%size of the $r$th pile $A^{[s]}_{r}$ is small, and in fact $A^{[s]}_{r} = A_0(1-p_nq_n)^{r}$ a.a.s. In other words, after $r$ steps in the $s$-threshold process the pile size is close to the expected size assuming exponential decay. This while at the same $r_n$ is small enough that:
%\begin{itemize}
%	\item With $s_n=q_n$, the pile sizes $A^{[s_n]}_1,\dotsc,A^{[s_n]}_{r_n}$ (weakly) overestimate the sizes of $r$ consecutive piles in the solitaire. This is true for any $r_n$ by Lemma~\ref{lem:combinatorial}(i).
%	\item With
%	\begin{equation}
%	\label{eq:s}
%	s_n=q_n(1+o(1)),
%	\end{equation}	
%	the pile sizes $A^{[s_n]}_1,\dotsc,A^{[s_n]}_{r_n}$ (weakly) underestimate the sizes of $r$ consecutive piles in the solitaire.	
%\end{itemize}
%In order to accomplish this, we must have
%$p_n r_n\to 0$ as $n\to\infty$. However, $p_n r_n$ may not tend to zero too fast. We shall
%require
\begin{equation}
\label{eq:pr1}
(p_n r_n)^2\frac{p_n q_n^2 n}{\log{n}} = \frac{p_n^3q_n^2 n r_n^2}{\log{n}} \rightarrow\infty \text{ as }n\rightarrow\infty.
\end{equation}
(Recall from \eqref{eq:npq2/logn} that $p_n q_n^2 n/{\log{n}} \to\infty$.)
However, since $r_n$ is a positive integer for any $n$, if $p_n\not\to 0$ we cannot have 
$p_n r_n \to 0$, but will see that $r_n=1$ suffices in the case $p_n\not\to 0$. In other words, we will require
\begin{equation}
\label{eq:pr2}
p_n(r_n-1)\rightarrow 0 \text{ as }n\rightarrow\infty.
\end{equation}

To accomplish (II) we shall see that $s=q$ will suffice for the overestimation and $s=q(1+2pr)=q(1+o(1))$ for the underestimation. 

One way of choosing $r_n$ such that \eqref{eq:pr1} and \eqref{eq:pr2} are fulfilled is 
\begin{equation}
\label{eq:rn}
r_n = \left\lceil \rho_n^{-1/3} p_n^{-1} \right\rceil \text{ where }\rho_n = \frac{p_n q_n^2 n}{1+\log n}.
\end{equation}
This choice fulfills \eqref{eq:pr2} since
$(r_n-1)p_n < r_np_n \le \rho_n^{-1/3} \rightarrow 0$.
That \eqref{eq:pr1} is fulfilled is easily verified:
\[
(p_n r_n)^2\frac{p_n q_n^2 n}{\log{n}} > \rho_n^{-2/3} \frac{p_n q_n^2 n}{\log{n}} = 
\frac{(p_nq_n^2 n)^{1/3}}{(1+\log n)^{-2/3}\log n} >
\left( \frac{p_nq_n^2 n}{\log n} \right)^{1/3} \to \infty
\]
as $n\to\infty$ by \eqref{eq:npq2/logn}.

Our next lemma, Lemma~\ref{lem:rchunk}, will bound the probability $P'$ that an initial pile of size $I_n:=O(pqn)$ will, after $r_n$ moves in an $s_n$-threshold process, deviate from the expected size assuming exponential decay, when $s_n=q_n(1+o(1))$.

Since the number of piles is $\approx(p_n q_n)^{-1}$, the number of $r$-chunks is $\approx(p_n q_n r_n)^{-1}$. When using Lemma~\ref{lem:rchunk} we need the bound $P'$ to hold for \emph{each} chunk during \emph{all} $M$ moves (where $M$ is given by \eqref{eq:DM}), specifically $P'M/(p_nq_nr_n) \to 0$ as $n\to\infty$. The probability in Lemma~\ref{lem:rchunk} is therefore bounded by
$o(p_n q_n r_n/M) = o(p_n^2 q_nr_n/n^2)$. This is also why the pile size deviation $\varepsilon np_nq_n$
is scaled with the number of chunks, resulting in the deviation $(\varepsilon p_n q_n n)(p_n q_n r_n) =  \varepsilon p_n^2 q_n^2 n r_n$.

\begin{lemma}
\label{lem:rchunk}
	Let $(p_n)_n$ and $(q_n)_n$ be real sequences such that $0<p_n,q_n\le 1$ and
	$p_n q_n \rightarrow 0 \text{ as }n\rightarrow\infty$.
	%\begin{equation}
	%\label{eq:thm_stoch_exp_rho}
	%\rho_n:=\dfrac{np_nq_n^2}{\log(1/q_n)} \rightarrow\infty \text{ as } n\rightarrow\infty.
	%\end{equation}
	For each $n$, let also
	$B_n \sim \Bin(F_n s_n,(1-p_n)^{r_n})$ where
	$(F_n)_n$ and $(s_n)_n$ are real sequences such that
	\begin{equation}
	\label{eq:thm_stoch_exp_As}
	F_n=O(p_n q_n n) \quad \text{and}\quad s_n=q_n(1+o(1)),
	\end{equation}
	and $F_n s_n$ is an integer for any $n$. Let also
	%$r_n=\dfrac{\sqrt{\log np_n}}{p_nq_n\sqrt{np_n}}$,
	$(r_n)_n$ be the sequence of positive integers in \eqref{eq:rn}.
	%\begin{equation}
	%\label{eq:thm_stoch_exp_pr}
	%p_n r_n=\frac{f(\rho_n)}{\sqrt{\rho_n}}\rightarrow 0 \text{ as }n\rightarrow\infty
	%\end{equation}
	%for some positive function $f$ with $f(x)\rightarrow\infty$ as $x\rightarrow\infty$,
	%$f:\mathbb{R}\to\Rpos$,
	%and $p_n r_n $ where $f(x)\rightarrow\infty$ slower than $\sqrt{x}$ as $x\rightarrow\infty$,
	%where $o$ and $O$ are with respect to $n\rightarrow\infty$.
	Then, for all $\varepsilon>0$ we have
	\[
	%P := P\left(\big| B + A(1-s) - A(1-pq)^r \big| > \varepsilon n p^2 q^2 r \right) = o(p^2 q r/n^2)
	P\left(\big| B_n + F_n(1-s) - F_n(1-p_nq_n)^{r_n} \big| > \varepsilon p_n^2 q_n^2 n r_n \right) = o(p_n^2 q_n r_n/n^2).
	\]
	%where the index $n$ is omitted on $A$, $B$, $p$, $q$, $r$ and $s$.
\end{lemma}
\begin{proof}
	Let us abbreviate $F=F_n$, $B=B_n$, $p=p_n$, $q=q_n$, $r=r_n$ and $s=s_n$. Thus, we want to prove that
	\[
	P := P\left(\big| B + F(1-s) - F(1-pq)^r \big| > \varepsilon p^2 q^2 nr \right) = o(p^2 q r/n^2).
	%P\left(\big| B_n + A_n(1-s) - A_n(1-p_nq_n)^r_n \big| > \varepsilon n p_n^2 q_n^2 r_n \right) = o(p_n^2 q_n r_n/n^2).
	\]
	We first note that the expected value $E(B) = Fs(1-p)^r$. Using the triangle inequality
	$|B+F(1-s)-F(1-pq)^r| \le |B-E(B)| + |E-F(1-pq)^r+F(1-s)|$ we obtain
	\[
	P \le P(|B-E(B)| > \varepsilon p^2 q^2 nr - |E(B)-F(1-pq)^r+F(1-s)|).
	\]
	By the Chernoff bound \eqref{eq:chernoff}
	%the Chernoff bound $P(|X-\mu| > \gamma) \le 2\exp(-\gamma^2/(3\mu))$ for $0<\gamma<\mu$ (for any binomially distributed $X$ with $E(X)=\mu$),
	we get
	\begin{equation}
	\label{lem:exp}
	P \le 2\exp\left(- \frac{(\varepsilon p^2 q^2 nr - |E(B)-F(1-pq)^r+F(1-s)|)^2}{3E(B)} \right).
	\end{equation}
	For the indices $n$ for which $r_n>1$ we have $rp \le 2(r-1)p \rightarrow 0$
	and hence
	\begin{equation}
	\label{eq:pqr}
	(1-p)^r = 1-pr + o(pr) \quad \text{and}\quad (1-pq)^r = 1-pqr + o(pqr).
	\end{equation}
	%\footnote{In fact, since each $r_n$ is an integer,
	%	$|(1-p)^r| = \left|1 - pr + \binom{r}{2}p^2 - \binom{r}{3}p^3 + \dotsb +(-1)^r p^r \right| < 1-pr+p^2r^2+\dotsb = 1-pr+o(pr).$}
	For the indices $n$ for which $r_n=1$, the relations in \eqref{eq:pqr} are trivially true.
	
	This means
	\begin{align*}
	|E(B)-&F(1-pq)^r+F(1-s)| = |Fs(1-p)^r-F(1-pq)^r+F(1-s)| \\
	&= |Fs(1-pr+o(pr))-F(1-pqr+o(pqr))+F(1-s)| \\
	&= F\Big( pr(q-s) + s\cdot o(pr) + o(pqr) \Big) \\
	&= o(Fpqr) = o(p^2 q^2 nr). & \text{(by }\eqref{eq:thm_stoch_exp_As}\text{)}
	\end{align*}
	%By \eqref{eq:thm_stoch_exp_s}, this equals $o(Apqr)$ which by the assumption on $A$ equals
	%$o(np^2q^2r) = np^2q^2r \cdot o(1)$.
	Thus the numerator in \eqref{lem:exp} can be written 
	%Thus the term $\varepsilon np^2 q^2 r$ in the numerator in \eqref{lem:exp} dominates, so it can be written
	$[(\varepsilon+o(1)) p^2 q^2 nr]^2$.
	%Since $0<p\le 1$ we have $(1-p)^r = O(1)$, and 
	By the assumptions in \eqref{eq:thm_stoch_exp_As}, the denominator %$3E=3As(1-p)^r$
	in \eqref{lem:exp} can be written
	\[
	3E(B) = 3Fs(1-p)^r = 3\cdot O(pqn) \cdot q(1+o(1)) \cdot O(1) = O(p q^2 n). 
	\]
	Putting these together, the bound \eqref{lem:exp} on $P$ can be written
	\[
	-\frac{1}{\log P} = O\left( \frac{O(p q^2 n)}{[(\varepsilon+o(1))p^2 q^2 nr]^2} \right) = O\left( \frac{1}{p^3 q^2 n r^2} \right) = o\left( \frac{1}{\log n} \right)
	%P \le 2 \exp \left( -\frac{ [(\varepsilon+o(1))np^2 q^2 r]^2 }{O(npq^2)} \right)
	%= O(\exp(-\varepsilon^2 np^3q^2r^2)).
	%= O(pqr)
	\]
	%\begin{align*}
	%P & \le 2 \exp \left( -\frac{ [(\varepsilon+o(1))np^2 q^2 r]^2 }{O(npq^2)} \right) \\
	%  & = O(\exp(-\varepsilon^2 np^3q^2r^2)) \\  
	%  & = O(pqr) & 
	%\end{align*}
	%where \eqref{eq:pr1} was used in the last step. 
	%Since $\log n = o(np^3q^2r^2)$ (by \eqref{eq:pr1}), we get
	%$-\log P = o(\log n)$.
	where \eqref{eq:pr1} was used in the last step.
	Since $pqnr\to\infty$ (also by \eqref{eq:pr1}) and $pqr\to 0$ (by \eqref{eq:pq} and \eqref{eq:pr2}), we have 
	$\frac{1}{pqr} = o(n)$ and hence $\log\frac{1}{pqr} = o(\log n)$. Therefore
	\[
	-\frac{1}{\log P} = o\left( \frac{1}{\log n + \log\frac{1}{pqr}} \right) = 
	o\left( \frac{1}{\log\frac{n}{pqr}} \right).
	\]
	From this follows
	\[
	\log P=o\left(\log\frac{pqr}{n}\right) = o\left(\log\frac{p^2qr}{n^2}\right),
	\]
	%\[
	%-\log P = o\left(\log\frac{n}{pqr}\right)
	%= o\left(\log\frac{n^2}{p^2 qr}\right),
	%\]
	thus $P=o(p^2 qr/n^2)$.
	%= O(\frac{p^2qr}{n^2})$.
	%
	%By \eqref{eq:thm_stoch_exp_rho}, $np^3q^2r^2 = f(\rho)\log(np)$. Thus, $P = O((np)^{-\varepsilon^2 f(\rho)})<<O(1/\sqrt{np})$.
	%If $r = \frac{\sqrt{\log np}}{pq\sqrt{np}}$, then $P \le \exp(-\varepsilon^2 O(\log np))$.
\end{proof}

%Recall that we are considering the configuration $\alpha^{(D+M)}$ after $D+M$ moves from the initial configuration $\alpha^{(0)}$. Lemma~\ref{lem:M} asserts that the number of piles in $\alpha^{(D+M)}$ is less than $M$. We will now investigate the sizes of these $<M$ piles. We shall do this $r$ piles at a time, using the $s$-threshold process and Lemma~\ref{lem:rchunk} in each interval of $r$ piles. When the $s$-threshold process for $r$ piles is used in this recursive way, the resulting process (considering \emph{any} number of moves) will be referred to as the $(r,s)$-\emph{threshold process}.

Note that Lemma~\ref{lem:rchunk} concerns an  $s$-threshold process, i.e.\ only $r$ steps.
In other words, it asserts that
\begin{equation}
\label{eq:lemrchunkform}
P\left(\big| A_{r+1} - A_1(1-p_nq_n)^{r_n} \big| > \varepsilon p_n^2 q_n^2 n r_n \right) =
o(p_n^2 q_n r_n/n^2),
\end{equation}
where $A_1=O(p_n q_n n)$ is the first pile size in an $r$-chunk and $A_{r+1}=(1-s_n)A_1+B_n$ the last
(see Figure~\ref{fig:exp}). However, the deviation and the probability were chosen in such a way that they can be added
over all $r$-chunks. This is done in Lemma~\ref{lem:mellan} which 
bounds the probability for deviation for the \emph{entire} union process.
Specifically, we will show that, for any $C>0$, the piles in $\Gamma$ formed at most $\frac{C}{pq}$ moves ago, i.e.\ $\Gamma_k$ for $k \le \frac{C}{pq}$, will follow an exponential decay a.a.s. The sizes of the piles formed more than $\frac{C}{pq}$ moves ago ($k>\frac{C}{pq}$) will be shown to be sufficiently small to be close enough to the tail in the exponential limit shape.

%Mellanlemma
\begin{lemma}
	\label{lem:mellan}
	Let $U_1,U_2,\dotsc$ be the pile sizes in an $(r_n,s_n)$-union process corresponding to $\mathscr{B}(n,p_n,q_n)$, where the initial pile size is $U_1 = O(p_n q_n n)$, and $r_n$ is given by 
	\eqref{eq:rn} and $s_n=q_n(1+o(1))$. Let $M=\lceil n^2/p_n \rceil$. Then
	\[
	\forall C,\varepsilon>0: \forall k<\tfrac{C}{p_n q_n} : P(|U_{k+1}-U_1(1-p_n q_n)^k| > \varepsilon p_n q_n n) = o(1/M) = o(p_n/n^2).
	\]	
\end{lemma}
\begin{proof}
	As in the proof of Lemma~\ref{lem:rchunk}, for the simplicity of notation we do not indicate in $p$, $q$, $r$ and $s$
	the dependence on $n$.
	%let us abbreviate $p=p_n$, $q=q_n$, $r=r_n$ and $s=s_n$.
	Let $C,\varepsilon>0$ and $\varepsilon'=\varepsilon/C$.
	%Fix a positive integer $k<\frac{C}{pq}$.
	By the triangle inequality,
	\begin{align*}
	|U_{k+r+1} - U_1(1&-pq)^{k+r}| \\
	& \le |U_{k+1} - U_1(1-pq)^{k}|(1-pq)^r + |U_{k+r+1} - U_{k+1}(1-pq)^r| \nonumber \\
	& \le |U_{k+1} - U_1(1-pq)^k| + |U_{k+r+1} - U_{k+1}(1-pq)^r|.
	\end{align*}
	Lemma~\ref{lem:rchunk} is now applicable for the first pile in each $r$-chunk
	(since $U_1 \ge U_2 \ge \dotsb$ and $U_1 = O(p_n q_n n)$), so by its formulation \eqref{eq:lemrchunkform},
	%By Lemma~\ref{lem:rchunk}
	$|U_{k+r+1} - U_{k+1}(1-pq)^r| < \varepsilon' p^2 q^2 r n$
	with probability ${1-o(p^2qr/n^2)}$. Thus,
	\begin{equation}
	\label{eq:rec}
	|U_{k+r+1} - U_1(1-pq)^{k+r}| < |U_{k+1} - U_1(1-pq)^{k}| + \varepsilon' p^2 q^2 r n
	\end{equation}
	with probability $1-o(p^2qr/n^2)$.	
	%The induction hypothesis: \marginpar{How to use this?}
	%\[
	%|A_k-A_0(1-pq)^k| < \varepsilon' A_0 pqk
	%\]
	%with probability $1-kpq\cdot o(1)$ as $n\rightarrow\infty$.
	We now note that the first term in the right hand side has the same form as the left hand side, only shifted with $r$ piles. Thus, by induction it follows that, for any positive integer $d$, we have
	\[
	|U_{dr+1}-U_1(1-pq)^{dr}| < d \varepsilon' p^2 q^2 r n
	%< d\varepsilon' np^2 q^2 = C\varepsilon' npq.
	\]
	with probability $1-o(p^2qr/n^2)$.
	%we can use \eqref{eq:rec} recursively for $k=dr$ where $d=1,2,\dotsc$ to obtain
	%$k=0,r,2r,\dotsc,(\eta-1)r$, where $\eta=\lfloor \tfrac{C}{pqr} \rfloor$ to obtain
%	\begin{align*}
%	&\text{for }k=0: && |A^{[s]}_r-A^{[s]}_0(1-pq)^r| < |A^{[s]}_0 - A_0| + \varepsilon' np^2 q^2 r < \varepsilon' np^2 q^2 r, \\
%	&\text{for }k=r: && |A^{[s]}_{2r}-A^{[s]}_0(1-pq)^{2r}| < |A^{[s]}_r - A_0(1-pq)^r| + \varepsilon' np^2 q^2 r < 2\varepsilon' np^2 q^2 r, \\
%	&\text{for }k=2r: && |A^{[s]}_{3r}-A^{[s]}_0(1-pq)^{3r}| < |A^{[s]}_{2r} - A_0(1-pq)^{2r}| + \varepsilon' np^2 q^2 r < 3\varepsilon' np^2 q^2 r, \\
%	&\dotsb & \\
%	&\text{for }k=(\eta-1) r: && |A^{[s]}_{\eta r}-A^{[s]}_0(1-pq)^{\eta r}| < \eta\varepsilon' np^2 q^2 r
%	< \tfrac{C}{pqr}\varepsilon' np^2 q^2 = C\varepsilon' npq,
%	\end{align*}
%	each with probability $1-o(p^2qr/n^2)$. 
	Thus, adding the probabilities for deviation for $k=r,2r,\dotsc,\eta r$, where $\eta=\lfloor \tfrac{C}{pqr} \rfloor$ we get
	\begin{align}
	\label{eq:rmult}
	P\bigl(\forall k\in\{r,2r,\dotsc,\eta r\} & : |U_{k+1}-U_1(1-pq)^k|>\eta \varepsilon' p^2 q^2 rn \ge \varepsilon pqn \bigr) \nonumber \\
	&= \eta \cdot o(p^2qr/n^2) = o(p/n^2).
	\end{align}
	We have thereby proved the claim in the lemma for $k=r,2r,\dotsc,\eta r$. If $k$ is not a multiple of $r$, suppose
	$dr < k < (d+1)r$ for some positive integer $d$. 	
	Then, since $pqr\to 0$ as $n\to\infty$ (which follows from \eqref{eq:npq2/logn} and \eqref{eq:rn}), we have $(1-pq)^r=1-pqr+o(pqr)$ and hence
	%and $A_0=O(npq)$, we have
	\[
	|U_1(1-pq)^{(d+1)r}-U_1(1-pq)^{dr}| = O(pqn)(1-pq)^{dr}|pqr+o(pqr)| < \varepsilon pqn.
	\]
	The lemma then follows by \eqref{eq:rmult} and the fact that ${U}_{dr} \le {U}_k \le {U}_{(d+1)r}$.
	%enclosing $A_{k}$ between $A_{dr}$ and  $A_{(d+1)r}$.
\end{proof}	

\newpage
\section{Proof of Theorem~\ref{thm:main_new}}
Below follows the proof of Theorem~\ref{thm:main_new}, stated in Section~\ref{sec:exp-limit}.

\begin{proof}
First, as in the previous section, let us consider $\mathscr{B}(n,p_n,q_n)$ as a process on $\mathcal{W}(n)$ rather than on $\mathcal{P}(n)$, and let $\alpha^{(0)} \in \mathcal{W}(n)$ be the weak composition representing the initial configuration of cards in the solitaire. Let also $M$ and $D$ be given by \eqref{eq:DM}.

%By Lemma~\ref{lem:M} applied on $\alpha^{(D)}$, all piles present in $\alpha^{(D)}$ have
%disappeared in $\alpha^{(D+M)}$ a.a.s., and by Lemma~\ref{lem:D} we have
%\[
%N(\alpha^{(D)}), N(\alpha^{(D+1)}), N(\alpha^{(D+2)}), \dotsc, N(\alpha^{(D+M)}) \in o(qn).
%\]

%\paragraph{Thus, during these $M$ moves the newly formed pile will have size $\lambda_1 \sim npq$} So after $D+M$ moves, all original piles (after the first $D$ moves) have disappeared. During this whole $M$ process
%the newly formed pile has always been of size $npq$, since the number of piles has always been bounded by $nq$.
%
%After $D$ moves (or any power of $n$ moves) 
%
%Thus, after $D+M$ moves, all piles that were present after $D$ moves are consumed (a.a.s.) and all current piles are ``fresh'' in the sense that they were formed during these $M$ moves (a.a.s.). By Lemma~\ref{lem:M}, during these $M$ moves, the number of \emph{candidate} cards has consistently been $o(nq)$, and thus the number of \emph{picked} cards has been on average $npq$ in each move. In Lemma~\ref{lem:variance} we prove that the actual number of picked cards in each of these $M$ moves doesn't deviate from $npq$ with more than $\varepsilon$ a.a.s.

%By Lemma~\ref{lem:M} the number of piles in $\alpha^{(D+M)}$ is less than $M$.

Let $(r_n)_n$ be the sequence of positive integers given by \eqref{eq:rn} and let $(s_n)_n$ be the sequence $s_n=q_n(1+2p_n r_n)$.
By Lemma~\ref{lem:M} applied on $\alpha^{(D)}$, all piles present in $\alpha^{(D)}$ have
disappeared in $\Gamma := \alpha^{(D+M)}$ a.a.s. Let $\Gamma_k = \alpha^{(D+M)}_k$ for $1\le k \le M$ be the number of cards in the pile that was formed $k$ moves ago. By Lemma~\ref{lem:variance}, each of these piles had size $O(np_nq_n)$ a.a.s.\ when they were formed. Let $F_n:=O(np_nq_n)$ be a sequence such that $F_n s_n$ is an integer for each $n$. Let $0<\varepsilon<1$ and choose $C_n$ such that
$C_n > \frac{p_n q_n \log\varepsilon}{\log(1-p_n q_n)}$.

Let $\check{U}_1,\check{U}_2,\dotsc$ be the pile sizes in the $(r_n,s_n)$-union process with initial pile size $\Gamma_1$.
Using the fact that $p_n q_n r_n \to \ 0$, it is a straightforward computation to show that
$s_n=q_n(1+2p_nr_n)$ implies $(1-q_n)((1-p_n r_n)^{r_n}-\varepsilon p_n q_n r_n) > 1-s_n$
and therefore also
\[
%\label{eq:works}
(1-q_n)\big(A(1-p_nr_n)^{r_n}-\varepsilon A p_nq_nr_n \big) > (1-s_n)A \quad \text{for any }A>0.
\]
%in other words, with this choice of $s$, the last pile $A_r$ in an $(s,A_0)$-threshold process satisfies $|A_r-A_0(1-pq)^r|<\varepsilon A pqr$.
By Lemma \ref{lem:rchunk}, the probability that $\check{U}_1(1-p_nr_n)^{r_n} - \varepsilon \check{U}_1 p_nq_nr_n < \check{U}_{r_n+1}$
%(i.e.\ $|A_r-A_0(1-pr)^r|=\varepsilon A_0 pqr$)
is $P_1:=1-o(p_n^2 q_nr_n/n^2)$. Thus, with probability $P_1$ we have $(1-q_n)\check{U}_{r_n+1} > (1-s_n)\check{U}_1 \ge \check{U}_1-\lceil s\check{U}_1 \rceil$ so by Lemma~\ref{lem:combinatorial}(ii), the pile sizes $\check{U}_1,\check{U}_2,\dotsc,\check{U}_{r_n+1}$ in the first $r$-chunk of the $(r_n,s_n)$-process underestimate the pile sizes $\Gamma_1,\Gamma_2,\dotsc,\Gamma_{r_n+1}$ with probability $P_1$.
In the next chunk, we have a new absolute threshold $s\Gamma_{r_n+2}=s\check{U}_{r_n+2}$. Since $\Gamma_{r_n+2} \le \Gamma_1$,
we have $(1-q_n)\check{U}_{2r_n+2} > (1-s_n)\check{U}_{r_n+2}$ with probability at least $P_1$, making Lemma~\ref{lem:combinatorial}(ii)
applicable also for the second chunk to conclude that $\check{U}_{r_n+2},\dotsc,\check{U}_{2r_n+2}$ underestimate $\Gamma_{r_n+2},\dotsc,\Gamma_{2r_n+2}$ with probability at least $P_1$. Continuing in the same manner for the first $C_n(p_n q_n r_n)^{-1}$ chunks, we conclude that the $(r_n,s_n)$-union process underestimates the solitaire with high probability:
\[
P(\check{U}_{k} > \Gamma_k \text{ for all } k<\frac{C_n}{p_n q_n})<(1-P_1)C_n(p_n q_n r_n)^{-1}=o(p_n/n^2).
\]

Let $\widehat{U}_1,\widehat{U}_2,\dotsc$ be the pile sizes in the $(r_n,q_n)$-union process with initial pile size $\Gamma_1$. By Lemma~\ref{lem:combinatorial}(i) (with $s_n=q_n$), the $(r_n,q_n)$-union process surely overestimates the solitaire in each chunk.

Taking the results for the $(r_n,s_n)$-union process and the $(r_n,q_n)$-union process together we have
\[
P(\check{U}_k \le \Gamma_k \le \widehat{U}_k \text{ for all } k<\frac{C_n}{p_n q_n}) > 1-o(p/n^2).
\]
%By Lemma~\ref{lem:combinatorial}, the $q_n$-threshold process underestimates the pile sizes $A_k$ and
%the $q_n(1+o(1))$-threshold process overestimate them, i.e.\ $A^{[q_n]}_k \le A_k \le A^{[s_n]}_k$.
Now, applying Lemma~\ref{lem:mellan} to both the pile sizes $\check{U}_k$ and to the pile sizes $\widehat{U}_k$ and using the squeeze theorem, we obtain
\[
\forall \varepsilon>0: \forall k<\tfrac{C_n}{p_n q_n} : P(|\Gamma_k-\Gamma_1(1-p_n q_n)^k| > \varepsilon np_nq_n) < o(p_n/n^2).
\]
Thus, the probability that $|\Gamma_k-\Gamma_1(1-p_nq_n)^k| < \varepsilon np_nq_n$ for the first $\tfrac{C_n}{p_nq_n}$ piles
throughout \emph{all} $M$ moves from $\alpha^{(D)}$ to $\alpha^{(D+M)}$ is $1-M\cdot o(1/M) = 1-o(1)$.

For piles $k>\tfrac{C_n}{p_n q_n}$, the exponential decrease (with decay factor $1-pq$) in pile size will yield piles smaller than $np_nq_n(1-p_nq_n)^{\frac{C_n}{p_n q_n}} < \varepsilon np_nq_n$ (by our choice of $C_n$).
%\to e^{-C_n}np_n q_n$.
Thus, the pile sizes themselves are below $\varepsilon np_nq_n$. 

%Thus, since all piles we see now are among the $M$ piles, and most of them have become tiny ($<\varepsilon npq$), and the freshest ($c/pq$ piles) follow the exponential shape, we have a shape which is at most $\varepsilon npq$ away from the exponential limit shape $e^{-x}$.

In summary; playing sufficiently many moves of $\mathscr{B}(n,p_n,q_n)$, the resulting composition diagram will a.a.s.\ be arbitrarily close to the boundary diagram of the composition $\alpha$ where $\alpha_k=np_nq_n(1-p_nq_n)^{k-1}$ for all $k=1,2,\dotsc$. The corresponding boundary function is
$\rescaled{\alpha}(x) = np_nq_n(1-np_nq_n)^{\lfloor x \rfloor}$. The corresponding rescaled
boundary function, with the given scaling factor $a_n=(p_n q_n)^{-1}$, is
\[
\rescaleds{a_n}{\alpha}(x) = (1-p_nq_n)^{\frac{x}{p_nq_n}} \to e^{-x}
\]
since $p_nq_n \to 0$ as $n\to\infty$.

Setting $m:=D+M$, and letting $\pi_n^m$ denote the probability distribution on $\mathcal{W}(n)$ for $\alpha^{(m)}$, we have
\[
\lim_{n\rightarrow\infty}\pi_n^m\left\lbrace \alpha\in\mathcal{W}(n): |\rescaleds{a_n}{\alpha}(x)-e^{-x}|<\varepsilon\right\rbrace=1,
\]
for all $\varepsilon>0$ and all $x>0$, in accordance with \eqref{eq:limitshape_def}.
By virtue of Lemma~\ref{lem:reordering_cor}, the same limit shape holds when configurations in the solitaire $\mathscr{B}(n,p_n,q_n)$ are represented by partitions $\mathcal{P}(n)$.

%---

%Note that Theorem~\ref{thm:main_new} does not refer to the stationary distribution $\pi_{n,p_n,q_n}$ of the Markov chain $(\lambda^{(0)},\lambda^{(1)},\dotsc)$, but to the probability distribution of $\lambda^{(m)}$ for a finite $m$. However, as a simple consequence of Theorem~\ref{thm:main_new}, the stationary distribution also has the limit shape $e^{-x}$.

Since $\pi_{n,p_n,q_n}$ is the stationary distribution of the Markov chain $(\lambda^{(0)},\lambda^{(1)},\dotsc)$, if we start with a partition $\lambda^{(0)}$ sampled from $\pi_{n,p_n,q_n}$ and play $m$ moves, the resulting partition $\lambda^{(m)}$ will also be sampled from $\pi_{n,p_n,q_n}$. Thus, the theorem follows by choosing $\lambda^{(0)}$ as a stochastic partition sampled from the stationary distribution.
\end{proof}

\section{Conjectures}
\label{sec:conj}
Recall that Theorem~\ref{thm:main_new} was proved with $\mathscr{B}(n,p_n,q_n)$ being considered
a process on $\mathcal{W}(n)$, and by virtue of Lemma~\ref{lem:reordering_cor} it also holds in
$\mathcal{P}(n)$. We imposed the condition $\frac{p_n q_n^2 n}{\log n}\to\infty$.
Here we conjecture that the weaker condition $p_n q_n^2 n \to\infty$ suffices in order for
Theorem~\ref{thm:main_new} to hold in $\mathcal{P}(n)$. 
%With a technique that takes care of the sorting, we believe this is possible.

\begin{conjecture}
	Theorem~\ref{thm:main_new} holds also when the condition $np_nq_n^2/{\log n}\to\infty$ is replaced by the weaker condition $np_nq_n^2\to\infty$.
\end{conjecture}

The reason for this conjecture can be 
understood by considering the example $q_n=1$ and $p_n n=\log(\log n)$. For this example
it is easy to prove that there is no limit shape when sorting is not performed. 
Since $q_n=1$, the number of picked cards in each move and thus the expected size
of a new pile is $\Bin(n,p_n)$ with expected value $np_n$ and standard deviation $\sigma \approx \sqrt{np_n}$.
A pile of size $np_n$ will after $1/p_n$ moves have the expected size $np_n(1-p_n)^{1/p_n} \to e^{-1}np_n$
as $n\to\infty$.

Thus, the probability for a ``visible'' deviation (i.e.\ greater than $d=\sqrt{np_n}$ standard deviations)
is $P(\text{deviation}\ge d\sigma)
%\approx e^{-d^2}
= e^{-np_n}$, so for $1/p_n$ piles,
%(whose sizes are approximately independent, at least when picking few cards per pile),
the probability for a visible deviation anywhere is
$\frac{e^{-np_n}}{p_n} = \frac{n}{\log n \log(\log n)} \to \infty$ as $n\to\infty$, i.e.\ the 
expected number of such large deviations tends to infinity as $n$ tends to infinity.
This makes it impossible to achieve a convergence in probability towards a limit shape.
However, from simulations we have reason to believe that the process converges towards a limit shape
when sorting is performed.

Further, recall from Section~\ref{sec:regimes} the other regimes $np_nq_n^2\to 0$ and $np_nq_n^2\to C$ from some constant $C>0$. We conjecture that the limit shapes in the $p_n$-random $q_n$-proportion Bulgarian solitaire in these regimes are the same as in the deterministic $q$-proportion Bulgarian solitaire developed in \cite{EJS}.

\begin{conjecture}
	If  $p_n q_n^2 n \to 0$ as $n\to\infty$, the limit shape of the $p_n$-random $q_n$-proportion Bulgarian solitaire is triangular.
\end{conjecture}

\begin{conjecture}
	If  $p_n q_n^2 n \to C$ as $n\to\infty$ for some constant $C>0$, the limit shape of the $p_n$-random $q_n$-proportion Bulgarian solitaire is a piecewise linear shape that depends on the value of $C$.
\end{conjecture}

\providecommand{\bysame}{\leavevmode\hbox to3em{\hrulefill}\thinspace}
%\providecommand{\MR}{\relax\ifhmode\unskip\space\fi MR }
%\providecommand{\MRhref}[2]{%
%	\href{http://www.ams.org/mathscinet-getitem?mr=#1}{#2}
%}
%\providecommand{\href}[2]{#2}

\end{document}